\documentclass[11pt]{amsart}
\usepackage{amssymb,amsmath,amsthm}
\usepackage{comment}
\usepackage{xcolor}
\usepackage{tikz}
\usepackage{enumitem}
\usepackage[utf8]{inputenc}      
\usepackage[T2A,T1]{fontenc}     
\usepackage[russian,english]{babel} 
\usetikzlibrary{patterns}
\newtheorem{theorem}{Theorem}
\newtheorem{lemma}[theorem]{Lemma}
\newtheorem{corollary}[theorem]{Corollary}
\newtheorem{proposition}[theorem]{Proposition}
\theoremstyle{definition}
\newtheorem{definition}[theorem]{Definition}
\newtheorem{question}[theorem]{Question}
\newtheorem*{theorem*}{Theorem}
\newtheorem*{corollary*}{Corollary}
\newtheorem{openproblem}[theorem]{Open Problem}
\newtheorem{remark}[theorem]{Remark}
\newtheorem{example}[theorem]{Example}

\newcommand{\mbbr}{\mathbb R}
\newcommand{\eps}{\varepsilon}
\def\diam{\operatorname{diam}}
\def\Lip{{\rm Lip}}
\def\supp{{\rm supp\,}}
\def\dist{{\rm dist\,}}

\usepackage{scalerel}[2014/03/10]
\usepackage[usestackEOL]{stackengine}
\def\intavg{\,\ThisStyle{\ensurestackMath{%
    \stackinset{c}{0\LMpt}{c}{0\LMpt}{\SavedStyle-}{\SavedStyle\phantom{\int}}}%
    \setbox0=\hbox{$\SavedStyle\int\,$}\kern-\wd0}\int}

\numberwithin{theorem}{section} \numberwithin{equation}{section}
\usepackage{graphicx}
\newcommand{\cube}{\scalebox{0.5}{$\mathrm{Q}$}}
\newcommand{\dube}{\scalebox{0.5}{$\mathrm{D}$}}

\thanks{R.M.\ was supported by Research Council of Finland Centre of Excellence in Randomness and Structures, Project number 364210.}
\author{Behnam Esmayli}
\address{Department of Mathematical Sciences, P.O.~Box 210025, University of Cincinnati, Cincinnati, OH~45221-0025, U.S.A.{\tt esmaylbm@ucmail.uc.edu}}
\author{Riddhi Mishra}
\address{Department of Mathematics and Statistics, University of Jyv\"askyl\"a, P.O. Box 35 (MaD), FI-40014, University of Jyv\"askyl\"a, Finland. {\tt riddhi.r.mishra@jyu.fi }}

\title{On Removable Sets for Weighted Sobolev Functions}
\begin{document}

\begin{abstract}We give sufficient geometric conditions, not involving capacities, for a compact null set to be removable for the Sobolev functions on weighted $\mathbb R^n$, defined as the closure of smooth functions in the weighted Sobolev norm. Our porosity conditions are in terms of suitable coverings by cubes. The weights are assumed to be doubling and satisfy a Poincar\'e inequality, which includes, but is not equal to, the famous class of Muckenhoupt weights. Our proofs use ideas and techniques from the theory of analysis on metric spaces.  
\end{abstract}
\maketitle

\section{Introduction}
A possible meta definition of \emph{removability} could be as follows. Suppose $\mathcal A(U)$ is a well-defined class of functions, say, with certain regularities, whenever we fix an open set $ U \subset \mathbb R^n$.  We ask that $A(U)$ be a subset of $A(U')$ in a canonical way whenever $U \subset U'$.

We say a compact set $E \subset U$ is \emph{removable} for $\mathcal{A}(U)$ if every function in $\mathcal{A}(U\setminus E)$ can be canonically identified with a function in $\mathcal{A}(U)$; possibly after being extended onto $E$ in some canonical way. If $E$ is removable for all open sets $U \supset E$, we say that $E$ is removable for $\mathcal A$, without referring to any particular $U$. Often it happens that removability of $E$ for one open set implies its removability for all open sets.

Here are some elementary examples. No nonempty set $E$ is removable for $\mathcal A=C^\infty$. Because if $x_0\in E$, then $u(x)=|x-x_0|^\beta$ for $\beta<1$ is in $C^\infty(\mathbb R^n \setminus E)$, but it does not agree with any $C^\infty$ function defined on $\mathbb R^n$. On the other hand, as long as $\mathbb R^n \setminus E$ is dense in $\mathbb R^n$, the set $E$ will be removable for the class of $\alpha$-H\"older continuous functions, for any fixed $0<\alpha\le 1$. This is because every $\alpha$-H\"older continuous function has a unique $\alpha$-H\"older continuous extension to the closure of its domain.

Between the two extremes, the questions of removability become delicate and involved. Removability has been extensively studied for harmonic functions, e.g.,  in \cite{Ahlf-Beur} and \cite{Hedberg:74}; for quasiconformal mappings, e.g.,  in \cite{gold-vod:ncp}, \cite{kaufman-wu}, and \cite{jmw}; and for Sobolev functions, e.g.,  in \cite{Hedberg:74}, \cite{jmw}, \cite{koskela:hyperplane}, \cite{jones-smirnov}, \cite{ntala}.

Note that we have not specified where the functions in $\mathcal A(U)$ are defined and this is important. For example, in \cite{jones-smirnov} and \cite{ntala} the definition of removability for Sobolev functions includes the assumption that the functions are continuous on $U$ and Sobolev on $U\setminus E$ where $E$ is to be the removable set. Removability questions for analytic functions also often include the apriori assumption of continuity on all of the domain.

In this paper we study the question of removability of sets for Sobolev functions on weighted $\mathbb R^n$, i.e.\, $\mathbb R^n$ equipped with a Borel measure $\mu$. The associated Sobolev spaces $H^{1,p}(\Omega;\mu)$ and $W^{1,p}(\Omega;\mu)$ will be thoroughly reviewed in Section~\ref{sec:prelim}.
\begin{definition}\label{def:remov-sob}
    Let $\Omega \subset \mathbb R^n$ be an open set. We say a compact $\mu$-null set $E \subset \Omega$ is removable for $H^{1,p}(\Omega;\mu)$ if every $u \in H^{1,p}(\Omega \setminus E;\mu)$ is also in $H^{1,p}(\Omega;\mu)$.
\end{definition}
 The definition of removability for $W^{1,p}$ is exactly the same except for changing ``H'' to ``W''.

Note that in both cases of $H^{1,p}$ and $W^{1,p}$, it is not necessary to (be able to) extend $u$ to $E$ or define it on $E$ (see Example~\ref{exmpl:x-x0}). Because $E$ is a null set, $u$, as it stands, is already in $L^p(\Omega;\mu)$. Since the $L^p$-norm is also indifferent to addition or omission of $E$, the question of removability is only about the regularity/differentiability.

In the case of $H^{1,p}$, removability is asking whether every function that is a limit, in the weighted Sobolev norm, of functions in $C^\infty(\Omega \setminus E)$ is also a Sobolev limit of functions in $C^\infty(\Omega)$.

In the case of $W^{1,p}$, we are asking whether every function in $L^p(\Omega;\mu)$ that is weakly differentiable on $\Omega \setminus E$ and its derivatives are in $L^p(\Omega;\mu)$ is also weakly differentiable on $\Omega$. Here the subtlety is the difference in the class of test functions for which the integration by parts \eqref{intg-parts} is required to hold.

In our main results, the measures will be of the form  $d\mu = wdx$ where $dx$ is the Lebesgue measure on $\mathbb R^n$ and $w$ is \emph{a weight} in the following sense.
 \begin{definition}[Weight]
A weight on $\mathbb R^n, n\ge 1$, is a nonnegative locally Lebesgue-integrable function.
 \end{definition} %
 We discuss weights and recall important well-known classes of them in Section~\ref{sec:weights}. Let us now see a basic example of removability.
\begin{example}\label{exmpl:x-x0}
    Consider $u(x)=|x-x_0|^\beta$, $\beta<0$, where $E=\{x_0\}$. We restrict the domain to $B=B(x_0,2)$. Suppose that the weight $w$ and $\beta<0$ are such that $u$ has finite weighted Sobolev norm defined in~\eqref{normm}. Since $u \in C^\infty(B \setminus E)$, we have $u \in H^{1,p}(B \setminus E;wdx)$. For $j \in \mathbb N$, we can replace $u$ on the ball $B(x_0,1/j)$ by a smooth function and obtain $u_j \in C^\infty(B)$. For reasonable weights $w$, the sequence $u_j$ will converge in $H^{1,p}(B;wdx)$ to $u$, proving that $u \in H^{1,p}(B;wdx)$. Observe that we cannot extend $u$ to $E$ continuously. However, this is not necessary for it to be a Sobolev function.
\end{example}
Once again, we emphasize that the question of removability of $E$ is not about extendability to $E$. As Example~\ref{exmpl:x-x0} demonstrates, even if we work with the so-called \emph{precise representatives} of functions, a function can have an essential discontinuity along the set $E$, and yet be in $H^{1,p}(\mathbb R^n;\mu)$.

Existing removability results are for the unweighted Sobolev spaces, where in fact $H^{1,p}$ and $W^{1,p}$ coincide. Vodopjanov-Goldstein~\cite{gold-vod:ncp} characterized removable sets for $W^{1,p}$ as null sets for suitable condenser capacities. An elegant proof of their theorem can be found in~\cite{kolsrud}.

However, capacities are practically impossible to calculate. Wu~\cite{jmw} and Koskela~\cite{koskela:hyperplane} found sufficient geometric conditions in terms of porosity for removability. These conditions do not use capacities.

The main result of this paper is in the spirit of \cite{jmw,koskela:hyperplane}  where we find a sufficient geometric condition for removability of sets for Sobolev functions. This condition is in terms of porosity and avoids any reference to capacity. The novelty is that now we work with weighted Sobolev spaces. Before delving into details, let us complete our literature review. 

In \cite{KoShTu}, the porosity condition of \cite{jmw,koskela:hyperplane} were adapted to yield a removability result in the metric space setting. Here the Sobolev space is in the sense of Newtonian-Sobolev functions defined via upper gradients. In \cite{KoShTu} removability is a corollary to a removability result for the Poincar\'e inequality. It was shown by Koskela~\cite{koskela:hyperplane} that removability for Sobolev functions is equivalent to removability for the Poincar\'e inequality, and the proof generalizes to the metric setting of~\cite{KoShTu}.

Other removability results for Sobolev functions and quasiconformal mappings can be found in~\cite{jones-smirnov}. Note that in~\cite{jones-smirnov} and~\cite{ntala}, the definition of removability of a set $E$ includes the assumptions that the functions are apriori continuous on $U$ in addition to being Sobolev on $U\setminus E$.

Removability of product Cantor sets $E \times F$ are studied in \cite{jmw}, \cite{ko-ra-zh} and \cite{bindini-rajala}. For weighted Sobolev spaces, but only with very specific choices of weights, removability has been studied, e.g.,  in \cite{Futamura-mizuta} and \cite{karak1}.

More recent results on removability can be found in \cite{bieg-warma}, \cite{ntala}, \cite{Panu23}, \cite{BBL23}, etc. The literature is extensive and we apologize for possible important omissions. We refer to the works mentioned and the references within them for more.

Removability questions for Sobolev spaces have close ties to, and were originally motivated by removability questions for quasiconformal mappings. Quasiconformal removability will not be addressed in this paper; we refer to \cite{Ahlf-Beur}, \cite{gold-vod:ncp}, \cite{aseev-syvcek}, \cite{kaufman-wu}, \cite{jmw}, and \cite{jones-smirnov}, and references therein.

We are ready to describe the main results of this paper about removability of sets for the Sobolev space $H^{1,p}(\Omega;\mu)$. Fix a compact set $E\subset \mathbb R^n$. Suppose that for each $k \in \mathbb N$, there is a covering $\mathcal Q_k$ of $E$ by pairwise disjoint cubes such that each cube $Q \in \mathcal Q_k$ contains a ring $R\subset Q$ along its perimeter that is away from $E$. The ratio of the thickness of the ring $R$ to the side length of the cube $Q$ is denoted by $\alpha_{\cube}$.

We will say that $E$ is $(s,p)$-porous, $s > p \ge 1$, if the measures and side lengths of the cubes and their rings in the (nested) sequence of coverings $\{\mathcal Q_k\}_{k=1}^\infty$ satisfy a particular divergence condition. See Section~\ref{sec:remov} for the details.

We say that a weight $w$ on $\mathbb R^n$ is $p$-admissible if the measure $d\mu=wdx$ is doubling and satisfies a $p$-Poincar\'e inequality (Section~\ref{sec:dblg-ppi-prelim}). The main result of our paper is the following.
\begin{theorem*}[\textbf{A}]
   Suppose $w$ is a $p$-admissible weight on $\mathbb R^n$, where $p\ge 1$, $n\ge 2$. If $E\subset \mbbr^n$ is $(s,p)$-porous for some $s>p$, then it is removable for $H^{1,s}(\mathbb R^n;wdx)$.  
\end{theorem*}
By appealing to the well-known open-ended property of the Poincar\'e inequality due to Keith and Zhong~\cite{Keith-Zhong:ann}, we deduce a removability result for~$H^{1,p}(\mathbb R^n;wdx)$:
\begin{theorem*}[\textbf{B}]
    Suppose $w$ is a $p$-admissible weight on $\mathbb R^n$, where $p > 1$, $n\ge 2$. Then there exists $\eps_0 >0$ such that for every $0<\eps \le \eps_0$, every $(p,p-\eps)$-porous set is removable for $H^{1,p}(\mathbb R^n;wdx)$.
\end{theorem*}

Theorem~(A) and Theorem~(B) are re-stated and proved as Theorem~\ref{jy1} and Theorem~\ref{thm:jyu:-1}, respectively.

Let us provide the sketch of our proof of Theorem~A. The proof is inspired by the proof in \cite{jmw} but requires many new tools since we work with arbitrary $p$-admissible weights. Fix $u \in H^{1,s}(\mathbb R^n \setminus E; \mu)$. Under the $(s,p)$-porosity condition on $E$, we must show that $u \in H^{1,s}(\mathbb R^n; \mu)$. Here is the rough sketch of the argument (compare to Example~\ref{exmpl:x-x0}):
\begin{enumerate}
    \item \textit{Surround the exceptional set:} From porosity, at scale $k$, we cover the set $E$ with the union of (small) disjoint cubes $Q$ such that a ring $R$ near its boundary, of relative thickness $\alpha=\alpha_{\cube}$, is away from $E$.
    \item \textit{Re-define the function:} We re-define $u$ in each $Q$ by Sobolev-extending the restriction of $u|_R$ to $Q$ (and replace the original values there). We thus obtain functions $u_k \in H^{1,p}(\mathbb R^n;\mu)$. We retain good control over the Sobolev norm of this extension.
    \item \textit{Recover $u$ as a Sobolev limit:} The porosity condition on $E$ is finetuned so that we can prove that $\{u_k\}_k$ has a Cauchy subsequence in $H^{1,p}(\mbbr^n;\mu)$ and hence a limit in $H^{1,p}(\mbbr^n;\mu)$. But the limit is just the original function $u$ since the $\mu$-measure of $\cup \mathcal Q_k$ converges to zero.
    \item \textit{Upgrade to $H^{1,s}$:} We have shown that $ u \in H^{1,p}(\mbbr^n;\mu)$ and we know that $u$ and $|\nabla u|$ are in $L^s(\mathbb R^n;\mu)$ by the assumptions. These do imply that $u \in H^{1,s}(\mathbb R^n;\mu)$, but not at all trivially. We prove this in Proposition~\ref{P1p=H1p}.
\end{enumerate}
In \cite{jmw}, the extension in step~(2) is achieved by using (a)~absolute continuity of Sobolev functions on a.e.\ line parallel to the axes, (b)~Fubini's theorem and (c)~a radial affine extension; which are not available in the presence of weights. We construct our extension by using the $p$-admissibility assumption. Also, step~(4) is obvious in \cite{jmw} as they work with $W^{1,p}$.

It is worth giving an example that illustrates the fact that step~(2) above is not just about finding any Sobolev extension, but rather the right one.
\begin{example}\label{exmp:bad-ext}
Consider the unweighted $\mathbb R^1$, i.e.\ $w\equiv 1$. Let $E=\{0\} \subset \mathbb R^1$. Let $u \equiv 0$ on $\mathbb R^1 \setminus E$. For $0<\delta <1$, let $u_\delta$ be the extension to $\mathbb R^1$ of $u|_{\mathbb R^1\setminus (-\delta,\delta)}$ given by $1-\frac{1}{\delta}|x|$ on $(-\delta,+\delta)$ and zero elsewhere. The functions $u_\delta$ are in $H^{1,1}(\mathbb R^1;dx)$, because they are Lipschitz, and their Sobolev norm is bounded by a constant independent of $\delta$. Moreover, $u_\delta(x) \to u(x)$ for almost every $x$. However, no subsequence of $u_\delta$ converges in $H^{1,1}(\mathbb R^1;dx)$ to $u$. Thus, this choice of extensions will not show that $u$ is a Sobolev function on $\mathbb R^1$, although, obviously, $u$ is in $H^{1,p}(\mathbb R^n;dx)$ for all $p\ge 1$.
\end{example}
Unlike in Example~\ref{exmp:bad-ext}, the extension that we construct in Section~\ref{sec:Proof-of-once-reflect} strongly reflects the behavior of the function near the set $E$, i.e.\ on the rings $R$. Here we make the most significant use of the $p$-Poinca\'e inequality assumption.

\subsection*{Outline of the paper}
In Section~\ref{sec:prelim} we recall definitions regarding weights, doubling, and $p$-Poincar\'e inequality. We then define the Sobolev spaces $H^{1,p}$ and $W^{1,p}$. We show why $H^{1,p}$ is the more natural space on a weighted $\mathbb R^n$, at least for removability questions (Theorem~\ref{singletonnotremov}). We also prove the local nature of removability. Finally, we discuss density of Lipschitz functions in $H^{1,p}$.

In Section~\ref{sec:extension-lemmas} we study the question of Sobolev extensions from the ring to the cube. This is a key step in the proof of the main removability results. Lemma~\ref{lem:reflect-one-time} is the technical backbone of the entire argument and we postpone its proof to Section~\ref{sec:Proof-of-once-reflect}.

In Section~\ref{sec:remov} we prove the main theorems about removability of porous sets. Section~\ref{sec:equiv-por} contains sufficient conditions for $(s,p)$-porosity. In Section~\ref{sec:examples}~we apply our porosity condition to show removability of certain product Cantor sets.

In Section~\ref{sec:Proof-of-once-reflect} we prove Lemma~\ref{lem:reflect-one-time}, and in Section~\ref{sec;appndx} we prove a result needed in the step~(4) of the proof sketch above.

\subsection*{Notation} Cubes will always have sides parallel to the coordinate axes. For a cube $Q$ and $\lambda >0$, $\lambda Q$ stands for the cube concentric to $Q$ and of side length $\lambda$ times that of $Q$. They will be open unless specified otherwise. Diameter and side length of $Q$ are denoted by $\diam Q$ and $\ell(Q)$, respectively. We denote the characteristic function of a set $A$ by $\chi_A$.

On $A \subset \mathbb R^n$, the set of all Lipschitz, $C^1$-regular, and smooth functions  will be denoted by $\Lip (A)$, $C^1(A)$, and $C^\infty(A)$, respectively. We write $\Omega \subset \subset \Omega'$ to mean that $\overline{\Omega}$ is compact and $\overline{\Omega} \subset \Omega'$. We write, e.g., $u \in \Lip_{loc}(\Omega')$ if $u \in \Lip(\Omega)$ for every $\Omega \subset \subset \Omega'$.

Measure to us means an \emph{outer measure}. It is Borel if Borel sets are measurable. When $0<\mu(A)<\infty$, we write 
$$
\intavg_A u\, d\mu:=\frac{1}{\mu(A)} \int_A u \, d\mu.
$$
we also use $u_{A}$ for the same quantity, especially when $A$ is a cube.

We have avoided indexing when possible. So, e.g., we use $\{u_k\}_k$ rather than $\{u_k\}_{k=1}^\infty$. Similarly, when $\mathcal{Q}$ is a collection of, say, cubes, we write $ \sum_{\mathcal{Q}}$ (union) rather than $\sum_{Q \in \mathcal{Q}}$, and $\cup \mathcal{Q}$ rather than  $\cup_{Q \in \mathcal{Q}} Q$.

We adhere to the standard usages of $\lesssim$ and $\approx$.

\section{Weights and Weighted Sobolev Functions}\label{sec:prelim}
\subsection{Doubling measures and Poincar\'e inequality}\label{sec:dblg-ppi-prelim}
We will adhere to the conventions in the book~\cite{Evans-Gariepy} regarding measure theory terminology. In particular, a measure means \emph{an outer measure}, and a Borel measure is one where all Borel sets are measurable.
\begin{definition}[doubling]
    We say a Borel measure $\mu$ on $\mathbb R^n$ is doubling if there exists a constant $C$ such that
\begin{equation}
    \label{dbln}
    \mu(2Q) \le C \mu(Q), \; \text{for every cube $Q \subset \mathbb R^n$,}
\end{equation}
and that for one cube, hence for all cubes, $0 < \mu(Q) < \infty$.
\end{definition}
Doubling can be equivalently stated with balls in place of cubes, i.e.\
$$
    \mu(B(x,2r)) \le C \mu(B(x,r)), \; \text{for every $x \in \mathbb R^n$ and every $r>0$.}
$$
On $\mathbb R^n$, doubling measures are $\sigma$-finite and $\mu(\mathbb R^n)=\infty$. Moreover, there exist constants $\delta' \ge n $, and $ 0 < \delta \le n$, and $C_1$ and $C_2>0$ such that
\begin{equation}
    \label{dbln-homg}
    C_2\Big(\frac{r}{\rho}\Big)^{\delta'} \le \frac{\mu(B(y,r))}{\mu(B(x,\rho))} \le C_1\Big(\frac{r}{\rho}\Big)^\delta ,
\end{equation}
for all $x \in \mathbb R^n$, all $0 < r \le \rho$ and all $y \in B(x,\rho)$.

The first inequality follows simply from the doubling condition, and $\delta'$ cannot be smaller than the Assouad dimension, which is $n$. The last inequality uses doubling and the connectedness of $\mathbb R^n$, and $\delta$ cannot be more than the Hausdorff dimension, which is again $n$; see \cite[Chapter~13]{Hei:01}.

For these and more facts on doubling measures, including in arbitrary metric spaces, see \cite{Hei:01}. Note that in \cite{Hei:01}, $y$ is taken to be $x$, but because $\mu(B(x,\rho))$ is comparable to $\mu(B(y,\rho))$ for all $y \in B(x,\rho)$ with a constant that only depends on the doubling constant of $\mu$, the version above follows immediately.

\begin{definition}[$p$-Poincar\'e inequality]\label{def:p-PI}
    We say that a Borel measure $\mu$ on $\mathbb R^n$ satisfies a $p$-Poincar\'e inequality, $1 \le p <\infty$, if $0<\mu(Q)<\infty$ for all cubes $Q \subset \mathbb R^n$ and there exists $C>0$ such that
\begin{equation}
        \label{p-pi}
        \intavg_Q|u-u_{\cube}|\, d\mu \le C (\diam Q)\Big(\intavg_{Q} |\nabla u|^p\, d\mu\Big)^{1/p},
\end{equation}
for every cube $Q$ and every bounded  $u \in C^\infty(Q)$.
\end{definition}
A more relaxed definition of a $p$-Poincar\'e inequality requires that there exist $\lambda \ge 1$ such that
\begin{equation}
        \label{lambda-p-pi}
        \intavg_Q|u-u_{\cube}|\, d\mu \le C (\diam Q)\Big(\intavg_{\lambda Q} |\nabla u|^p\, d\mu\Big)^{1/p},
\end{equation}
for every cube $Q$ and every bounded smooth $u \in C^\infty(\lambda Q)$.

By Remark~\ref{reM:N1p=H1p} and a well-known result due to Haj\l{}asz-Koskela \cite{Haj:Ko:met}, the weaker version \eqref{lambda-p-pi} implies \eqref{p-pi}. For technical reasons we decide to have $\lambda =1$ as part of our definition. (For example, the proof of the Sobolev extension result in Section~\ref{sec:Proof-of-once-reflect} is more streamlined because we do not have to dilate the cubes.)
\begin{definition}[$p$-admissible]\label{def:admiss}
    We say a Borel measure on $\mathbb R^n, n\ge 1$, is $p$-admissible, $p \ge 1$, if it is doubling and satisfies a $p$-Poincar\'e inequality.
\end{definition}

\subsection{\protect \boldmath \texorpdfstring{$p$}{p}-Admissible weights}\label{sec:weights}
Recall that \textit{a weight} is a nonnegative locally Lebesgue-integrable function $w(x)$ on $\mathbb R^n$. We use $dx$ for the Lebesgue measure on $\mathbb R^n$, with $n\ge 1$ understood from the context.

Given a weight, the expression
$$
\mu(A):=\inf\Bigl\{\int_{A'}w\, dx: \text{$A'$ is Lebesgue-measurable, $A\subset A'$} \Bigr\}
$$
defines a Borel measure on $\mathbb R^n$ that is, clearly, absolutely continuous with respect to the Lebesgue measure. The Radon-Nikodym derivative of $\mu$ is Lebesgue-a.e.\ equal to $w$. Therefore, we encode the measure $\mu$ defined above by the notation $d\mu = wdx$.

We identify the measure with the weight. Thus, for instance, expressions such as \textit{a doubling weight} are self-explanatory. Our main results will require \emph{$p$-admissible weights}, i.e.\ weights such that $d\mu=wdx$ is $p$-admissible, i.e.\ $\mu$ is doubling and satisfies a $p$-Poincar\'e inequality (Definition~\ref{def:p-PI}).

Because doubling measures on $\mathbb R^n$ are $\sigma$-finite, if a doubling measure $\mu$ is absolutely continuous with respect to the Lebesgue measure, then by the Radon-Nikodym theorem there exists a weight $w(x)$ such that
$$
\mu(A)=\int_A w\, dx,
$$
(at least) for all Lebesgue-measurable $A \subset \mathbb R^n$. An interesting problem \cite[Open Problem A.18]{Bj:Bj:11} asks, basically, whether every $p$-admissible measure comes from a weight (See Theorem~\ref{all-admss-Ap} below for dimension $n=1$.):
\begin{openproblem}[\cite{Bj:Bj:11}]
    Is every $p$-admissible measure on $\mathbb R^n$, $n\ge 2$, absolutely continuous with respect to the Lebesgue measure?
\end{openproblem}
Regardless of the answer, the class of $p$-admissible weights on $\mathbb R^n$, is already quite vast. Here is a non-exhaustive list (assume $n\ge 2$):
\begin{itemize}[topsep=.5ex]
    \item Every Muckenhoupt $A_p$ weight, $p\ge 1$, is $p$-admissible \cite{Fa:Je:Ke:82}.
    \item Let $ -n < \gamma$. Then the weight $|x|^\gamma$ is $p$-admissible for all $p\ge 1$ \cite{Heikima}, \cite{J:Bj:01}.
    \item Let $F\subset \mathbb R^n$ be a Lebesgue-null set and suppose $\mathbb R^n \setminus F$ is a uniform domain. Then $w(x)=\textup{dist}(x,F)^\beta$ is $1$-admissible for all $\beta \ge 0$ \cite{Keith:Modu}.
    \item The weight $w(x_1,x_2)=|x_1|$ on $\mathbb R^2$ is $p$-admissible for every $p>1$ \cite{Heikima}.
    \item  If $w(x)$ is the Jacobian of a quasiconformal map $f\colon \mathbb R^n \to \mathbb R^n$, $n \ge 2$, then for every $1 \le p < n$, the weight $w^{1-p/n}$ is $p$-admissible \cite{Hei:Ko:94}. (\cite{J:Bj:01} for $p=1$.)
    \item  If $w(x)$ is a positive superharmonic functions on $\mathbb R^n$, then $w$ is $p$-admissible for all $p\ge 1$ \cite[Theorem~3.59]{Heikima}.
\end{itemize}
By H\"older's inequality, if a measure satisfies a $p$-Poincar\'e inequality, then it satisfies an $s$-Poincar\'e inequality for every $s > p$, which leads to the following: 
\begin{proposition}\label{prop:q-admiss}
    Every $p$-admissible measure (respectively, every $p$-admissible weight) is $s$-admissible for every $s>p$.
\end{proposition}
So, $1$-Poincar\'e inequality, respectively, $1$-admissibility, is the strongest.
\subsection{Sobolev spaces on weighted \texorpdfstring{$\mathbb R^n$}{Rn}}
There are multiple ways to define a Sobolev space of functions on (open subsets of) the metric measure space $(\mathbb R^n, |x-y|,\mu)$. The question of removability is, of course, very much dependent on this choice. We recall the two main classical approaches here. At a few instances, we will appeal to the deep fact \cite{Bj:Bj:11} that $H^{1,p}$ coincides with the so-called Newtonian-Sobolev space $N^{1,p}$ (Remark~\ref{reM:N1p=H1p}), but we will not delve into the details of the latter space. Our proofs definitely borrow heavily from the techniques of \emph{analysis on metric spaces} in other ways as well, which will be apparent to the experts.
\begin{remark}[standing assumptions]\label{rem:standing-assum}
We will henceforth assume, throughout this note, that $w$ is a $p$-admissible weight, $p\ge 1$, on $\mathbb R^n$, $n \ge 1$. Thus, the corresponding measure $\mu$ given by $d\mu =w dx$
    \begin{enumerate}[topsep=0.3ex]
        \item is doubling,
        \item satisfies a $p$-Poincar\'e inequality, and
        \item is absolutely continuous with respect to the Lebesgue measure.
    \end{enumerate}
Many of the results in individual subsections remain true if we drop certain of these assumptions. However, we will effectively need all of these properties in the main results. Absolute continuity is needed to prove the density of Lipschitz functions in $H^{1,p}$. So, we forgo any attempt at minimizing the assumptions in isolated individual parts. The interested reader shall check which assumptions are \emph{essentially} needed in each proof and which can be dropped.
\end{remark}
Fix an open set $\Omega \subset \mathbb R^n$. For $u  \in C^\infty(\Omega)$, we define
\begin{equation}\label{normm}
    \|u \|_{1,p}:= \left( \int_\Omega |u |^p\, d\mu +\int_\Omega |\nabla u |^p\, d\mu \right)^{1/p},
\end{equation}
which may be infinite. We have suppressed the reference to $\mu$ and $\Omega$ in notation, so, all future occurrences of this norm should be understood in the weighted sense with $\mu$ and $\Omega$ implicit in the context.

The space of $L^p$-integrable functions on $\Omega$ is denoted by $L^p(\Omega;\mu)$, and
$$
\|u\|_{L^p(\Omega;\mu)}:=\Bigl(\int_\Omega |u|^p\, d\mu\Bigr)^{1/p}, \quad \|\nabla u\|_{L^p(\Omega;\mu)}:=\Bigl(\int_\Omega |\nabla u|^p\, d\mu\Bigr)^{1/p},
$$
in whatever sense $\nabla u$ is understood.

We define $H^{1,p}(\Omega;\mu)$ to be the completion in the $\|\cdot \|_{1,p}$-norm of the linear space $\{u  \in C^\infty(\Omega): \|u \|_{1,p} < \infty\}$. The advantage of this definition is that it immediately produces a Banach space. However, it is not clear in what sense the elements in $H^{1,p}$ have a derivative, and whether this derivative is unique. We will study this space in more detail in Section~\ref{sec:sob}.

Alternatively, we can define a Sobolev space via the notion of \emph{weak differentiability}. For $j\in\{1,\cdots,n\}$, recall that we say a locally Lebesgue-integrable function $v_j$ on $\Omega$ is a weak $j$'th partial derivative of a locally Lebesgue-integrable function $u$ on $\Omega$ if
\begin{equation}\label{intg-parts}
\int_\Omega u \,\partial_j\varphi \, dx = -\int_\Omega v_j \,\varphi\, dx, \quad \text{for all $\varphi \in C^\infty_0(\Omega)$}.
\end{equation}
We say a (locally Lebesgue-integrable) function $u$ is weakly differentiable if all of its weak partial derivatives are well-defined.

We define $W^{1,p}(\Omega;\mu)$ to be the space of all weakly differentiable functions $u$ on $\Omega$ such that the right-hand side of \eqref{normm} is finite, where $\nabla u$ is understood as the weak gradient $\nabla u:=(v_1,\cdots,v_n)$ of $u$. The right-hand side of \eqref{normm} is well-defined since by the absolute continuity assumption, every Lebesgue-measurable function is $\mu$-measurable as well. By abuse of notation we again equip $W^{1,p}$ with the norm $\|\cdot \|_{1,p}$.

One advantage of the space $W^{1,p}$ is the availability of Fubini's theorem. (For example, it helps to easily prove, in the unweighted setting, removability of sets with zero Hausdorff-$(n-1)$ measure.) Moreover, the functions in $W^{1,p}$ already come with derivatives, which we can easily show to be unique, up to difference on a Lebesgue null set.

One disadvantage of $W^{1,p}$ is that, in general, even with rather nice weights, $W^{1,p}(\Omega;\mu)$ may not be a Banach space (Theorem~\ref{singletonnotremov}). Notice that the definition of weak differentiability does not involve the weight; it only uses the Lebesgue measure. In this sense, it is not natural that the weight plays (apparently) no role in the definition of differentiability of functions in $W^{1,p}$; it is featured only in the integrability requirement.

Both spaces have been studied in literature. The space $H^{1,p}$ is, for example, extensively studied in the monograph~\cite{Heikima} and in the appendix to~\cite{Bj:Bj:11}, where it was proved that $H^{1,p}$ agrees with the Newtonian-Sobolev space on $(\mathbb R^n,|x-y|,\mu)$. On the other hand, the space $W^{1,p}$ is the predominant space, for example in works of Kufner, e.g., \cite{kufner:book,kufner:survey-paper}. The reader must be careful with notation across references. For example, \cite{Bj:Bj:11} uses $W^{1,p}$ as notation for what we call $H^{1,p}$.

\subsection{\protect\boldmath \texorpdfstring{$H \neq W$}{H neq W}}
The space $W^{1,p}(\Omega;\mu)$ is a subset of $L^p(\Omega;\mu)$ by definition, and the space $H^{1,p}(\Omega;\mu)$ can be identified with a subset of $L^p(\Omega;\mu)$ (Lemma~\ref{limfunc}). So, a natural question is their relation to each other. Whenever we have the equality $W^{1,p}=H^{1,p}$, we can combine the techniques and prove stronger removability results for these Sobolev functions.

In the unweighted $\mathbb R^n$, i.e.\ when $d\mu=dx$, the celebrated theorem of Meyers-Serrin tells that $W^{1,p}(\Omega;dx)=H^{1,p}(\Omega;dx)$, without any assumptions on $\Omega$. Similarly, it was proved in~\cite{Kilp:paper} that if $w$ is in the Muckenhoupt $A_p$ class, then $W^{1,p}(\Omega;wdx)=H^{1,p}(\Omega;wdx)$.

Muckenhoupt $A_p$ weights are a strict subset (when $n\ge 2$) of $p$-admissible weights. Unfortunately, for general $p$-admissible weights, the spaces $W^{1,p}$ and $H^{1,p}$ may fail to be equal. The example we give is not entirely new \cite[page 13]{Heikima}.
\begin{proposition}\label{HnotW}
    Fix $n \ge 2$ and the weight $w(x)=|x|^\gamma$, for some $\gamma > -n$. Then $w$ is a $p$-admissible weight on $\mathbb R^n$ for every $p\ge 1$. Let $B=B(0,1)$. Then, the function given by $u(x)=|x|^{-n+1}$ is not in $W^{1,p}(B;wdx)$, but it will be in $H^{1,p}(B;wdx)$ if $\gamma > n(p-1)$.
\end{proposition}
\begin{proof}
    The first claim about the admissibility is in \cite{Heikima}. 
    We have that $u \notin W^{1,1}(B;wdx)$ because it is not weakly differentiable because what would have had to be its gradient is of order $|x|^{-n}$ near origin, which is not locally integrable. The integration by parts fails for certain test functions. Note that this does not involve $\gamma$ or $p$ and that, interestingly, $u$ is absolutely continuous on almost all lines.

    Now observe that
    \begin{align*}
        \int_{B} |u(x)|^pw(x)dx \lesssim \int_{B} |\nabla u(x)|^pw(x)dx &\lesssim \int_{B} |x|^{-np+\gamma} dx,
    \end{align*}
    which will be finite if $-np+\gamma > -n$, or equivalently, if $\gamma>n(p-1)$.

    For each $j=2,3,\cdots$, we replace $u|_{B(0,1/j)}$ with a suitable smooth function and denote the resulting smooth function on $B$ by $u_j$. If chosen correctly, in light of the calculation above and the monotone convergence theorem, $u_j$ converge in $H^{1,p}$ to $u$. Hence, $u \in H^{1,p}(B;wdx)$.
\end{proof}
\begin{remark}
    The weight $w(x)=|x|^\gamma$  is a Muckenhoupt $A_1$ weight if and only of $-n < \gamma \le 0$. It is a Muckenhoupt $A_p$ weight, $p>1$, if and only if $-n < \gamma < n(p-1)$. For these claims see \cite{Heikima}. This explains the (sharp) threshold $\gamma > n(p-1)$ in Proposition~\ref{HnotW}, because for $A_p$ weights we would have $H^{1,p}=W^{1,p}$ \cite{Kilp:paper}.
\end{remark}
\begin{remark}
    Let $w(x)=|x|^{\gamma}$ be as in Proposition~\ref{HnotW}. Note that $\gamma > n(p-1)$ is exactly the cutoff so that $w^{{-1}/{(p-1)}}$ is locally integrable. The latter is a recurrent condition in the theory of weights. It implies, by H\"older's inequality, that $H^{1,p}(\Omega;wdx)$ embeds continuously in $H^{1,1}(\Omega;dx)$ (unweighted), for bounded $\Omega$. But then from $H^{1,p}=W^{1,p}$ in the unweighted class, we deduce that functions in $H^{1,p}(\Omega;wdx)$ are weakly differentiable.  For Muckenhoupt $A_p$ weights, $w^{{-1}/{(p-1)}}$ is in fact locally integrable.
\end{remark}
Proposition~\ref{HnotW} (after multiplying $u(x)=|x|^{-n+1}$ by a smooth cutoff function that vanishes outside of, say, $B(0,3)$) proves the following:
\begin{theorem}\label{singletonnotremov}
    For every $p\ge 1$ and $n\ge 2$ there exists a $p$-admissible weight $w$ on $\mathbb R^n$ such that
    \begin{itemize}[topsep=0.4ex]
        \item[(1)] $W^{1,p}(\mathbb R^n;wdx)$ is not a Banach space,
        \item[(2)] there exists $u \in H^{1,p}(\mathbb R^n; wdx)$ that is not weakly differentiable,
        \item[(3)] the set $\{0\}$ is not removable for $W^{1,p}$, i.e.\ there exists a function in $W^{1,p}(\mathbb R^n \setminus \{0\};wdx)$ that does not coincide, in a.e.\ sense, with any function in $W^{1,p}(\mathbb R^n;wdx)$.
    \end{itemize}
\end{theorem}
In light of Theorem~\ref{singletonnotremov}, especially part (3), we will work with the space $H^{1,p}$ rather than $W^{1,p}$. But remember that for Muckenhoupt $A_p$ weights the two spaces coincide on every open set \cite{Kilp:paper}.
\begin{remark}
    One may, as is the case in \cite{zhikov98}, consider the abstract completion of $W^{1,p}$ in the Sobolev norm. The resulting Banach space will always contain $H^{1,p}$. However, we are faced with the same issue: the new elements in the enlarged space may fail to be locally Lebesgue-integrable and/or weakly differentiable, as the example in Proposition~\ref{HnotW} illustrates.
\end{remark}
For the sake of completeness, let us emphasize that in dimension one, the situation is very different.
\begin{theorem}[\cite{JBj:Bu:Ke}]\label{all-admss-Ap}
   A measure $\mu$ is $p$-admissible on $\mathbb R^1$, $p\ge 1$, if and only if $\mu$ is absolutely continuous with respect to the Lebesgue measure and $d\mu = wdx$ for a Muckenhoupt $A_p$ weight $w$.
\end{theorem}
\begin{corollary}
    Let $\mu$ be a $p$-admissible measure, $p \ge 1$, on $\mathbb R^1$. Then for every open set $\Omega \subset \mathbb R^1$, we have $W^{1,p}(\Omega;\mu)=H^{1,p}(\Omega;\mu)$.
\end{corollary}

\subsection{The Sobolev space \protect\boldmath \texorpdfstring{$H^{1,p}$}{H1p}}\label{sec:sob}
Recall that we defined $H^{1,p}(\Omega;\mu)$ as the completion in the $\|\cdot \|_{1,p}$-norm of the linear space $\{u  \in C^\infty(\Omega): \|u \|_{1,p} < \infty\}$.

By its definition, $H^{1,p}(\Omega;\mu)$ is a Banach space. Objects in $H^{1,p}$ are, per definition, equivalence classes of Cauchy sequences of smooth functions. We wish to show that we can realize $H^{1,p}(\Omega;\mu)$ as a subset of $L^p$-integrable functions on $\Omega$.

Fix an element  $[u] \in H^{1,p}(\Omega;\mu)$. Take a Cauchy sequence $u_i \in C^\infty(\Omega)$, i.e.\ Cauchy in the $\|\cdot\|_{1,p}$-norm, that represents $[u]$. Then $u_i$ form a Cauchy sequence in $L^{p}(\Omega;\mu)$. Thus, by completeness of $L^{p}(\Omega;\mu)$, there exists a $u \in L^{p}(\Omega;\mu)$ to which $u_i$ converge in $L^p$.

Moreover, $\nabla u_i$ form a Cauchy sequence in $L^{p}(\Omega;\mu)^n$. Thus, there exists an vector-valued function in $L^{p}(\Omega;\mu)^n$, which we denote by $\nabla u$, to which $\nabla u_i$ converge in $L^{p}(\Omega;\mu)^n$.

The functions $u$ and $\nabla u$ depend only on $[u]$ and not on the representative sequence $u_i$. Thus we have proved that there is an embedding
\begin{equation}\label{embed-H1p}
    H^{1,p}(\Omega;\mu) \to L^{p}(\Omega;\mu) \times L^{p}(\Omega;\mu)^n,
\end{equation}
which we can turn into a continuous (or even isometric) one by equipping the target with a natural norm. 

The definition of $H^{1,p}$ works with any weight or even any Borel measure. However, in general, the gradient $\nabla u$ may not be uniquely determined by $u$, the first component of the embedding in~\eqref{embed-H1p} \cite{Fa:Je:Ke:82}. Under our assumptions of doubling and Poincar\'e inequality, the gradient will always be unique. The first proof of this fact is due to Semmes and was published in \cite{Hei:Ko:94}. (See also \cite[Chapter 20]{Heikima} where a shorter proof attributed to \cite{Fr-Haj-Kos} is provided.)

Using the uniqueness of gradients, we can finally identify $[u]$ with $u$, the first component of the embedding in~\eqref{embed-H1p}. So, finally, in this precise sense, we can, and will, view elements in $H^{1,p}(\Omega;\mu)$ as $L^p$-integrable functions.

Incidentally, we have proved the following useful characterization:
\begin{lemma}
\label{limfunc}
The space $H^{1,p}(\Omega;\mu)$ consists of all $u \in L^p(\Omega;\mu)$ such that there exists a $v\in L^{p}(\Omega;\mu)^n$ and a sequence $u_i \in C^\infty(\Omega)$ with the property that $\|u_i-u\|_{L^p(\Omega;\mu)} \to 0$ and $ \|\nabla u_i - v \|_{L^p(\Omega;\mu)} \to 0$.
\end{lemma}
Let us record some properties of functions in $H^{1,p}$. The first claim is an immediate consequence of Lemma~\ref{limfunc}.
\begin{lemma}\label{cor:a.e.=H1p}
    If $ v \in H^{1,p}(\Omega;\mu)$ and $u(x)=v(x)$ at $\mu$-a.e.\ $x\in \Omega$, then $ u \in H^{1,p}(\Omega;\mu)$.
\end{lemma}
For a $\mu$-measurable function $u$ on $\Omega$, we write $u \in H^{1,p}_{loc}(\Omega;\mu)$ if $ u \in H^{1,p}(\Omega';\mu) $ for every open set $\Omega' \subset \subset \Omega$.
\begin{proposition}\label{prop:H1p-loc+Lp=H1p}
    If $u \in H^{1,p}_{loc}(\Omega;\mu)$ and $u$ and $|\nabla u|$ are in $L^p(\Omega;\mu)$, then $u \in H^{1,p}(\Omega;\mu)$.
\end{proposition}
Proposition~\ref{prop:H1p-loc+Lp=H1p} will be used in proving the locality of removability. Its proof uses a partition of unity subordinate to a covering of $\Omega$ by an increasing sequence of compactly contained open subsets \cite[Lemma~1.15]{Heikima}.
\begin{proposition}[weak compactness]\label{lem:weak-conv}
    For every $p>1$, the space $H^{1,p}(\Omega;\mu)$ is \emph{sequentially weakly compact}, i.e., given any bounded sequence $\{u_j\}_j$ in $H^{1,p}(\Omega;\mu)$, there exist $u \in H^{1,p}(\Omega;\mu)$ and a subsequence $\{u_{j_k}\}_k$ such that $u_{j_k}$ converge weakly in $L^p(\Omega;\mu)$ to $u$ and $\nabla u_{j_k}$ converge weakly in $L^p(\Omega;\mu)^n$ to $\nabla u$, interpreted in terms of the coordinate functions.
\end{proposition}
This follows from standard functional analysis techniques and properties of $L^p$-spaces. Not much is special about $H^{1,p}$ here \cite[Theorem~1.31]{Heikima}.

\begin{lemma}\label{cor:p-PI-for-H}
The $p$-Poincar\'e inequality holds for functions in $H^{1,p}$: let $C>0$ be the constant in the definition of the $p$-Poincar\'e inequality i.e., in~\eqref{def:p-PI}, then for every open $\Omega \subset \mathbb R^n$, every $u \in H^{1,p}(\Omega;\mu)$, and every $Q \subset \subset \Omega$, we have
\begin{equation*}
        \intavg_Q|u-u_{\cube}|\, d\mu \le C (\diam Q)\Big(\intavg_{Q} |\nabla u|^p\, d\mu\Big)^{1/p}.
\end{equation*}
\end{lemma}
\begin{proof}
    By definition, there exists a sequence $u_i \in C^\infty(\Omega)$ that converges to $u$ in $ H^{1,p}(\Omega;\mu)$. In particular, $u_i$ converge in $L^p$ to $u$ and $\nabla u_i$ converge in $(L^p)^n$ to $\nabla u$. By H\"older's inequality, $u_i$ converge in $L^1(Q;\mu)$ to $u$.
    
    Since, $u_i$ are bounded and smooth on $Q$, by the assumption of $p$-Poincar\'e inequality, for each $i$
\begin{equation}\label{approx-seq-pi}
        \intavg_Q|u_i-(u_i)_{\cube}|\, d\mu \le C (\diam Q)\Big(\intavg_{Q} |\nabla u_i|^p\, d\mu\Big)^{1/p}.
\end{equation}
From the observations above, it is easy to see that
$(u_i)_{\cube} \to u_{\cube}$ and $u_i-(u_i)_{\cube}$ converge in $L^1(Q;\mu)$ to $u-u_{\cube}$. Thus, taking limit as $i \to \infty$, in~\eqref{approx-seq-pi}, completes the proof. 
\end{proof}
We will need a consequence of Poincar\'e inequality about averages.
\begin{lemma}\label{prop:avg-difference}
Let $\Omega$ be open, and $Q_0 \subset\subset \Omega$ and $Q_1$ be cubes such that $ Q_1 \subset Q_0 \subset \kappa Q_1$ for some $\kappa \ge 1$. Then, there exists a constant $C>0$ that depends only on $\kappa$ and the doubling and the Poincar\'e inequality constants of $\mu$ such that for every $u \in H^{1,p}(\Omega;\mu)$
   \begin{equation*}
       |u_{\cube_1}-u_{\cube_0}|\leq C \diam{(Q_0)}\left(\intavg_{Q_0}|\nabla u(x)|^pdx\right)^{\frac{1}{p}}.
   \end{equation*}
\end{lemma}
\begin{proof}
We have
     \begin{align*}
        |u_{\cube_1}-u_{\cube_0}| & \leq \intavg_{Q_1}|u(z)-u_{\cube_0}|\, d\mu\\
        &\le \frac{\mu(Q_0)}{\mu(Q_1)}\, \intavg_{Q_0}|u(z)-u_{\cube_0}|\, d\mu\\ 
        &\leq C \diam{(Q_0)}\left(\intavg_{Q_0}|\nabla u|^p \, d\mu\right)^{\frac{1}{p}},
    \end{align*}
    where in the last step we applied the doubling condition and the $p$-Poincar\'e inequality (Lemma~\ref{cor:p-PI-for-H}). In particular, $C$ depends only on $\kappa$ and the doubling and the Poincar\'e inequality constants of $\mu$.
\end{proof}
Lemma~\ref{prop:avg-difference} can be used to prove the following observation:
\begin{lemma}\label{n=1-nonremov}
If $d\mu=wdx$ supports a $p$-Poincar\'e inequality on $\mathbb R^1$, then no nonempty set is removable for $H^{1,p}(\mathbb R^1;\mu)$.    
\end{lemma}
\begin{proof}
    Suppose $x_0 \in E$. Let $u(x)=1$ if $x>x_0$ and $u(x)=0$ if $x<x_0$. Then clearly $u \in H^{1,p}((x_0-1,x_0+1)\setminus \{x_0\};\mu)$. If $u$ were to be equal to a function in $H^{1,p}((x_0-1,x_0+1);\mu)$, by Lemma~\ref{prop:avg-difference} and $\nabla u = 0$ a.e., we would have to have averages of $u$ on $(x_0-1,0)$ and $(x_0,x_0+1)$ to be equal, which is not the case.
    
    In view of locality of removability (Proposition~\ref{prop:locabc}), the proof is complete.
\end{proof}

In the proof of our main result (Theorem~\ref{jy1}), starting with a $u$ in $H^{1,s}(\mathbb R^n \setminus E;\mu)$, where $E$ is porous, at the first stage of the argument we will succeed in proving that $u \in H^{1,p}(\mathbb R^n;\mu)$, whereas we need to prove that $u \in H^{1,s}(\mathbb R^n;\mu)$, where $s>p$. This arises the following question:
\begin{question}\label{areULq}
    Let $1 \le p <s $. Suppose $ u \in H^{1,p}(\Omega;\mu)$ and we know that $u$ and $|\nabla u|$ are in $L^s(\Omega;\mu)$. Is it true that $u \in H^{1,s}(\Omega;\mu)$?
\end{question}
In Proposition~\ref{P1p=H1p} we give an affirmative answer to Question~\ref{areULq} when $d\mu=wdx$ for a $p$-admissible weight. We do not know the full generality under which the answer to Question~\ref{areULq} is positive.
\subsection{Locality of removability}\label{sec:locality-of-remov}
Recall the notion of removability of a compact set $E \subset \Omega$ for the Sobolev space $H^{1,p}(\Omega;\mu)$ from Definition~\ref{def:remov-sob}. Results in this section show the local nature of removability, so that, without loss of generality we may choose any convenient $\Omega$ as long as it contains the exceptional set $E$. Analogous statement should be true for $W^{1,p}$ functions.

We will need a gluing result:
\begin{proposition}\label{lem:glue}
    Suppose that $u_1 \in H^{1,p}(\Omega_1;\mu)$, $u_2 \in H^{1,p}(\Omega_2;\mu)$, and $u_1(x)=u_2(x)$ at $\mu$-a.e.\ $x \in \Omega_1 \cap \Omega_2$.  Define $u$ on $\Omega:=\Omega_1 \cup \Omega_2$ by $u(x)=u_1(x)$ on $\Omega_1$ and by $u(x)=u_2(x)$ on $\Omega_2\setminus\Omega_1$. Then $u$ is in $H^{1,p}(\Omega;\mu)$ and $\nabla u$ coincides with $\nabla u_1$, $\mu$-a.e.\ on $\Omega_1$ and with $\nabla u_2$, $\mu$-a.e.\ on $\Omega_2$.
\end{proposition}
\begin{proof}
 First, for $i=1,2$, fix $\Omega_i' \subset \subset \Omega_i$ and smooth $\varphi_i$ such that $\supp \varphi_i \subset \Omega_i$ and $\varphi_1+\varphi_2\equiv 1$ on $\Omega_1' \cup \Omega_2'$. For each $i=1,2$, let $\{u_{i,j}\}_{j=1}^\infty$ be a sequence of smooth functions such that $ u_{i,j} \to u_i$ in $H^{1,p}(\Omega_i;\mu)$. It now follows from standard calculations that
 $$
 u_j(x):=\varphi_1(x)u_{1,j}(x)+\varphi_2(x)u_{2,j}(x), \quad x \in \Omega_1' \cup \Omega_2'
 $$
 converges in $H^{1,p}(\Omega_1' \cup \Omega_2';\mu)$ to $u$. Here we have set $\varphi_i(x)u_{i,j}(x) =0 $ outside $\Omega_i$, and this yields smooth functions. The claim about the gradients also follows from the same calculations.

 Since we can exhaust $\Omega_i$ by such $\Omega_i'$, this proves that $u \in H^{1,p}_{loc}(\Omega;\mu)$. Now, Proposition~\ref{prop:H1p-loc+Lp=H1p} proves that $u \in H^{1,p}(\Omega;\mu)$, as desired.
\end{proof}
\begin{proposition}\label{prop:locabc}
    Fix a compact $\mu$-null set $E \subset \mathbb R^n$. The following are equivalent: 
    \begin{itemize}[topsep=0.5ex]
        \item[(a)] $E$ is removable for $H^{1,p}(\mathbb R^n;\mu)$,
        \item[(b)] $E$ is removable for $H^{1,p}(\Omega;\mu)$ for every open set $\Omega \supset E$,
        \item[(c)] $E$ is removable for $H^{1,p}(\Omega;\mu)$ for some open set $\Omega \supset E$.        
    \end{itemize}
\end{proposition}
\begin{proof}
    (a) $\Rightarrow$ (b): Assume (a) holds and fix an arbitrary $u \in  H^{1,p}(\Omega \setminus E;\mu)$. Fix a $ \theta \in C^\infty_0(\Omega)$ such that $\theta \equiv 1$ on
    $$
    \Omega':=\{ x \in \Omega: \dist (x,\partial \Omega) > \frac{1}{2} \dist (E,\partial \Omega) \}.
    $$
    Then, $\theta u \in H^{1,p}(\mathbb R^n \setminus E;\mu)$. Thus, by (a) assumed, $ \theta u \in H^{1,p}(\mathbb R^n;\mu)$. In particular, $u=\theta u \in H^{1,p}(\Omega';\mu)$.

    On the other hand, clearly, $u \in H^{1,p}(\Omega'';\mu)$, where
    $$
    \Omega'':=\{ x \in \Omega: \dist (x,\partial \Omega) < \frac{3}{4} \dist (E,\partial \Omega) \}.
    $$
    Thus, by Proposition~\ref{lem:glue}, $u \in H^{1,p}(\Omega;\mu)$. Since, $u \in  H^{1,p}(\Omega \setminus E;\mu)$ was arbitrary, we have proved (b).

    (b) $\Rightarrow$ (c) is a tautology.

    (c) $\Rightarrow$ (a): Assume (c) holds and fix an arbitrary $u \in  H^{1,p}(\mathbb R^n \setminus E;\mu)$. Then, trivially, $u \in  H^{1,p}(\Omega \setminus E;\mu)$. By (c) assumed, this implies that $u \in  H^{1,p}(\Omega;\mu)$. A similar argument to the above (i.e., choices of open sets that intersect along a very small neighborhood of $\partial \Omega$ and application of Proposition~\ref{lem:glue}) shows that $u \in H^{1,p}(\mathbb R^n;\mu)$. We have proved (a).
\end{proof}
By  compactness of $E$, Propositions~\ref{lem:glue} and~\ref{prop:locabc}, and similar arguments to the ones in their proofs one also proves:
\begin{proposition}
    Let $E \subset \mathbb R^n$ be a compact $\mu$-null set. Then $E$ is removable for $H^{1,p}$ if and only if every $x \in E$ has an open nonempty neighborhood $U_x$ such that $E \cap U_x$ is removable for $H^{1,p}(U_x;\mu)$.
\end{proposition}

We state a useful fact about the role of $p$ in removability. The analogous statement for $W^{1,p}$ is almost trivial, but for $H^{1,p}$ we will use the nontrivial result Proposition~\ref{P1p=H1p}.
\begin{proposition}\label{prop:q-remov}
    Every removable set for $H^{1,p}$ is removable for $H^{1,q}$ for every $q>p$.
\end{proposition}
\begin{proof}
    In light of the locality of removability, it is enough to consider, without loss of generality, the case where $E \subset Q_0=(0,1)^n$.
    
    Suppose $u \in H^{1,q}(Q_0\setminus E;\mu)$. By definition, there exist $\varphi_j \in C^\infty(Q_0\setminus E)$ that converge in $H^{1,q}(Q_0\setminus E;\mu)$ to $u$. By H\"older's inequality $\varphi_j$ converges to $u$ in the $\|\cdot\|_{1,p}$-norm. So, $u \in H^{1,p}(Q_0\setminus E;\mu)$. By the assumption, then $u \in H^{1,p}(Q_0;\mu)$. But $u$ and $|\nabla u|$ are in $L^q(Q_0;\mu)$ since $\mu(E)=0$. Now, the fact that $u \in H^{1,q}(Q_0;\mu)$ follows from Proposition~\ref{P1p=H1p}.
\end{proof}

\subsection{Density of Lipschitz functions in \protect \boldmath \texorpdfstring{$H^{1,p}$}{H1p}}\label{sec:Lip-dense}
Although smooth functions are by definition dense in $H^{1,p}$, in practice, working with them is more demanding. More flexible are the Lipschitz functions. Every Lipschitz function is Fr\'echet differentiable a.e.\ by Rademacher's theorem, it is weakly differentiable and and its  Fr\'echet derivative coincides with its weak derivative. Lipschitz function will play a key role in our construction of the extension operator. 

\begin{theorem}\label{thm:loclipRH}
Every locally Lipschitz function $u\colon \Omega \to \mathbb R$ is in $H^{1,p}_{loc}(\Omega;\mu)$ for every $p\ge 1$. Moreover, $\nabla u$ coincides with the classical derivative of $u$, which is also its weak derivative.
\end{theorem}
Theorem~\ref{thm:loclipRH} is proved via approximations by (the standard, unweighted) convolution. For details consult \cite[Lemma 1.11]{Heikima}. This is where we need the absolute continuity of $\mu$ with respect to the Lebesgue measure.
\begin{corollary}
     Claims of Lemmas~\ref{cor:p-PI-for-H} and~\ref{prop:avg-difference}, about Poincar\'e inequality and the differences of averages, hold for locally Lipschitz functions $u$ on $\Omega$.
\end{corollary}
\begin{corollary}\label{cor:lip-are-dense}
    The subset of $H^{1,p}(\Omega;\mu)$ consisting of locally Lipschitz functions is dense in $H^{1,p}(\Omega;\mu)$.
\end{corollary}

\section{Sobolev Extension from Ring to Cube}\label{sec:extension-lemmas}
In the rest of the paper, we assume $n \ge 2$, and $d\mu=wdx$ for a $p$-admissible weight on $\mathbb R^n$. 
 
The assumption $n \ge 2$ is justified by Lemma~\ref{n=1-nonremov}, which says that no nonempty set is removable in $\mathbb R^1$. Interestingly, it is Lemma~\ref{lem:half-Q-ring} below that fails for $n=1$ and explains why the proof in the next section will not apply in $\mathbb R^1$.

Suppose $Q\subset \mathbb R^n$ is a cube and $R \subset Q$ is a (uniform) ring along its perimeter. As it is without loss in generality for the ensuing analysis, we may imagine (figure~\ref{fig:big-ring})
\begin{align*}
    Q&=\{(x_1,\cdots,x_n) \in \mathbb R^n: \max_i |x_i| < \ell/2 \},  \\
    R&=\{(x_1,\cdots,x_n) \in Q: (1-\alpha)\ell/2 < \max_i |x_i| < \ell/2 \}.
\end{align*}

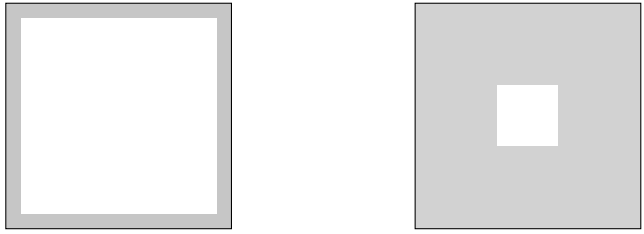
\begin{figure}[h]\label{fig:big-ring}
  \centering
\begin{minipage}{0.3\textwidth}
    \begin{tikzpicture}
    \centering
    \def\outerSize{3cm}
     \def\ringThickness{.2cm}
  
     \fill[gray!45] (0,0) rectangle (\outerSize,\outerSize);
  
     \fill[white] (\ringThickness,\ringThickness) rectangle 
               ({\outerSize-\ringThickness},{\outerSize-\ringThickness});
  
     \draw[thin] (0,0) rectangle (\outerSize,\outerSize);
    \end{tikzpicture}
\end{minipage}
\hspace{1cm}
\begin{minipage}{0.3\textwidth}
    \centering
    \begin{tikzpicture}
    \def\outerSize{3cm}
     \def\ringThickness{1.1cm}
  
      \fill[gray!35] (0,0) rectangle (\outerSize,\outerSize);
  
      \fill[white] (\ringThickness,\ringThickness) rectangle 
               ({\outerSize-\ringThickness},{\outerSize-\ringThickness});
  
      \draw[thin] (0,0) rectangle (\outerSize,\outerSize);
    \end{tikzpicture}
\end{minipage}
\caption{Examples of $Q$ with their ring $R$ (shaded) for small and large $\alpha$. For thin rings, we first extend to the adjacent ring of equal width, then repeat with their union.}
\end{figure}

These are precisely balls and annuli in $\mathbb R^n$ equipped with the supremum norm, but we wish to avoid that terminology and notation. We will write $\ell(Q):=\ell$, $\ell(R):=\alpha \ell,$ and $ \alpha_{\cube}:=\alpha$. The parameter $\alpha_{\cube} \in (0,1)$ is \emph{the relative thickness} of the ring and will play an important role. Note that $Q$ and $\alpha$ determine $R$ uniquely.

The aim of this section is to construct a bounded linear extension\footnote{We say $Eu$ is an extension of $u\colon A \to \mathbb R$ if $Eu(x)=u(x)$ at $\mu$-a.e.\ $x \in A$.} operator from (basically) $H^{1,p}(R;\mu)$ to $H^{1,p}(Q;\mu)$. The motivation, clarified in the next section, is that $Q$ will be one of the finitely many cubes that cover the null porous set $E$ that is supposed to be removable and $R\subset Q$ is a ring that is away from $E$. We wish to use the behavior of our Sobolev function on $R$ to study it in $Q$. This fulfills the step (2) in the outline in the Introduction.

We must mention the well-known extension theorems in \cite{Chua94,chua2006} for the weighted Sobolev functions (of higher orders as well), where the assumptions on the weight are doubling and a Poincar\'e inequality, as in our setting.

Unfortunately, the extension theorems in \cite{Chua94} and~\cite{chua2006} are not sufficient for our needs, at least without significant modifications. The norm of the extension operator goes to infinity when either $\alpha \to 0$ or $\ell(Q) \to 0$, both of which are cases that are very much relevant to us. This is expected of course. However, we need quite explicit quantitative bounds on the norm. Because the weight does not scale like the Lebesgue measure, we cannot argue with scalings.

What we will do, is that we will first extend from a ring \emph{only} to the adjacent ring of equal width, with an operator norm independent of $Q$ and $\alpha$. Then we repeat the extension with the union of the two rings on which we have the function defined now. Repeating inductively, we will succeed in ``filling in'' almost half of $Q$. We then end with one final extension to all of $Q$. Although again the norm of the extension from the original ring to $Q$ goes to infinity as $\alpha \to 0$, we will obtain a good enough explicit estimate on its growth (Theorem~\ref{extsob}).

By the density of Lipschitz functions (Corollary~\ref{cor:lip-are-dense}), in order to prove the extension theorem for $H^{1,p}$, it is sufficient to prove it for Lipschitz functions.

Continuing with notation from above, when $\alpha<1/2$, define
$$
2R:=\{(x_1,\cdots,x_n): (1-2\alpha)\ell/2 < \max_i|x_i| < \ell/2\}.
$$
\begin{lemma}\label{lem:reflect-one-time}
Let $\alpha_{\cube} < 1/2$. There exists a linear extension operator $E\colon \Lip(R) \to  \Lip(2R)$  that satisfies
\begin{equation*}
         \| Eu\|_{L^{p}(2R;\mu)} \le C_1 \|u\|_{L^{p}(R;\mu)}, \quad \| \nabla Eu \|_{L^{p}(2R;\mu)} \le  C_1 \|\nabla u\|_{L^{p}(R;\mu)},
\end{equation*}
where $C_1\ge 2$ is a constant that does not depend on $Q$, $\alpha$, or $u$.
\end{lemma}
Lemma~\ref{lem:reflect-one-time} is the technical backbone of the entire argument in this paper. We postpone its proof to Section~\ref{sec:Proof-of-once-reflect} in the interest of a smoother presentation and to emphasize that similar removability results would follow in other settings if one proves an analogue of such extension theorems.

As mentioned already, we will apply Lemmas~\ref{lem:reflect-one-time} to $2R$ and then to $2(2R)$, etc., until we extend to at least ``half of'' $Q$, Once there, we will use the next result to extend to all of $Q$:
\begin{lemma}\label{lem:half-Q-ring}
Let $\alpha_{\cube} \ge 1/2$. There exists a linear extension operator $E\colon \Lip(R) \to  \Lip(Q)$ that satisfies
\begin{equation*}
         \| Eu\|_{L^{p}(Q;\mu)} \le C_0 \|u\|_{L^{p}(R;\mu)}, \quad \| \nabla Eu \|_{L^{p}(Q;\mu)} \le  C_0 \|\nabla u\|_{L^{p}(R;\mu)},
\end{equation*}
where $C_0$ is a constant that does not depend on $Q$, $\alpha$, or $u$.
\end{lemma}

We continue by assuming Lemmas~\ref{lem:reflect-one-time} and \ref{lem:half-Q-ring}. The reader shall decide to read their proofs in Section~\ref{sec:Proof-of-once-reflect} before proceeding, if they wish so.

For  $0 < \alpha_{\cube} <1$, let $m \in \{0,1,2,\cdots\}$ be the smallest integer such that $2^m\alpha \ge 1/2$. Equivalently, it is determined by  $2^m \ell(R) < \ell(Q) \le 2^{m+1}\ell(R)$.
\begin{proposition}\label{prop:fullextlip}
Let $0 < \alpha_{\cube} <1$. There exists a linear extension operator $E\colon \Lip(R) \to  \Lip(Q)$
that satisfies
\begin{align*}
         \| Eu\|_{L^{p}(Q;\mu)} \le C_0C_1^m \|u\|_{L^{p}(R;\mu)}, \quad \| \nabla Eu \|_{L^{p}(Q;\mu)} \le  C_0C_1^m \|\nabla u\|_{L^{p}(R;\mu)},
\end{align*}
where $C_1\ge 2$ and $C_0>0$ are structural constant that do not depend  on $Q$, $\alpha$, or $u$, and $m$ is the integer given by $2^m \ell(R) < \ell(Q) \le 2^{m+1}\ell(R)$. In particular, $C_1^m < \alpha_{\cube}^{-\log_2C_1} \le C_1^{m+1}$.
\end{proposition}
\begin{proof}
    Apply Lemma~\ref{lem:reflect-one-time} successively $m$-times (possibly $m=0$) and then apply Lemma~\ref{lem:half-Q-ring}. Notice that $C_1$ does not change in the iterations.
\end{proof}

By approximation we can now extend Sobolev functions, but we need to bypass a small technicality concerning local Lipschitz versus Lipschitz class.
\begin{definition}
For an open set $\Omega$, let $H^{1,p}(\overline{\Omega};\mu) \subset H^{1,p}({\Omega};\mu)$ stand for the subspace consisting of $u$ such that there exists $\widetilde{u} \in H^{1,p}(\Omega';\mu)$, for some open $\Omega' \supset \overline{\Omega}$, such that $u(x)=\widetilde{u}(x)$ a.e.\ on $\Omega$.
\end{definition}
For the so-called extension domains, the two spaces will coincide. This includes Lipschitz domains, e.g.,  our $R$. But we will not need even this fact. For extension theorems for weighted Sobolev functions see \cite{Chua94,chua2006} and the references therein.
\begin{theorem}\label{extsob}
Let $0 < \alpha_{\cube} <1$. There exists a linear extension operator 
$
E\colon H^{1,p}(\overline{R};\mu) \to H^{1,p}(Q;\mu)
$
that satisfies
\begin{align*}
         \| Eu\|_{L^{p}(Q;\mu)} \le C_0C_1^m \|u\|_{L^{p}(R;\mu)}, \quad \| \nabla Eu \|_{L^{p}(Q;\mu)} \le  C_0C_1^m \|\nabla u\|_{L^{p}(R;\mu)},
\end{align*}
where $C_1\ge 2$ and $C_0>0$ are structural constant that do not depend on $u$, $Q$, or $\alpha$, and $m$ is the integer given by $2^m \ell(R) < \ell(Q) \le 2^{m+1}\ell(R)$. In particular, $C_1^m < \alpha_{\cube}^{-\log_2C_1} \le C_1^{m+1}$.
\end{theorem}
\begin{proof}\footnote{In abstract language, the proof is that Lipschitz functions are dense in $H^{1,p}$ and any bounded linear map on a subspace extends to a bounded linear map on its closure.}
Corollary~\ref{cor:lip-are-dense} implies that Lipschitz functions on $R$ are dense in $H^{1,p}(\overline{R};\mu)$. So, we fix a sequence of Lipschitz functions $u_j$ on $R$ such that $u_j \to u$ on $R$ in the Sobolev norm. In particular, they form a Cauchy sequence.

Let $E$ be the extension operator in Proposition~\ref{prop:fullextlip}; this double-use of the notation ``$E$'' will be justified in the course of the proof. Hence, $Eu_j$ are Lipschitz function on $Q$, $Eu_j = u_j$ on $R$ for each $j$, and for each $j$
\begin{equation*}
     \| Eu_j\|_{L^{p}(Q;\mu)} \le C_0C_1^m \|u_j\|_{L^{p}(R;\mu)}, \quad \| \nabla Eu_j \|_{L^{p}(Q;\mu)} \le  C_0C_1^m \|\nabla u_j\|_{L^{p}(R;\mu)}.
\end{equation*}

By the linearity and boundedness of $E$, $Eu_j$ form a Cauchy sequence in $H^{1,p}(Q;\mu)$. Thus, by the completeness of $H^{1,p}(Q;\mu)$, $Eu_j$ have a limit $\overline{u} \in H^{1,p}(Q;\mu)$.

Since $Eu_j = u_j$ on $R$ for each $j$, by standard measure theory facts we deduce that $\overline{u}=u$ $\mu$-a.e.\ on $R$. Moreover, by passing to the limit as $j \to \infty$ in the inequalities above, we get
    \begin{equation*}
         \| \overline{u} \|_{L^{p}(Q;\mu)} \le C_0C_1^m \|u\|_{L^{p}(R;\mu)}, \quad \| \nabla \overline{u} \|_{L^{p}(Q;\mu)} \le  C_0C_1^m \|\nabla u\|_{L^{p}(R;\mu)}.
\end{equation*}
It is now easy to see that $Eu:=\overline{u}$ is a linear extension operator that coincides with the operator from Proposition~\ref{prop:fullextlip} for Lipschitz $u$ (hence justifying the double-use of the notation $E$) and satisfies the claims of the theorem.
\end{proof}
Theorem~\ref{extsob} will be enough for our purposes, since $R$ will be away from the exceptional set $E$, and hence functions in $H^{1,p}(\mathbb R^n\setminus E;\mu)$ will be Sobolev on a slightly larger open set containing $\overline{R}$. 
\section{Removability of Porous Sets}\label{sec:remov}
Recall the definitions and notations of the ring $R \subset Q$ and the relative thickness $\alpha_{\cube} \in (0,1)$ associated to a cube $Q$ from Section~\ref{sec:extension-lemmas}.

Fix a compact $\mu$-null set $E \subset \mathbb R^n$. Suppose that for each $k=1,2,\ldots$, there exists a collection $\mathcal{Q}_k$ of finitely many cubes $Q$ such that
\begin{itemize}
    \item $E \subset \cup \mathcal{Q}_k$ and cubes in $\mathcal{Q}_k$ are pairwise disjoint, and
    \item for every $Q \in \mathcal{Q}_{k+1}$ there exists  $Q' \in \mathcal{Q}_{k}$ such that $Q \subset (1-\alpha_{\cube'})Q'$.
\end{itemize}
 It follows from these conditions that for every $k$ and every cube $Q \in \mathcal{Q}_k$, the ring $R \subset Q$ has a positive distance to $E$.
 
 We will use the notation
$$
c_1:=\log_2 C_1,
$$
where $C_1 \ge 2$ is the constant in the claim of Theorem~\ref{extsob}.

We say a compact $\mu$-null set $E \subset \mathbb R^n$ is $(s,p)$-porous\footnote{This notion does depend on $\mu$ and should be called $(s,p;\mu)$-porous, but for us $\mu$ will always be understood.} with parameters $s>p\ge 1$, if we can find collections $\mathcal{Q}_k$ of cubes as above that moreover satisfy:
$$
   \lim_{k\to \infty} \mu (\cup {\mathcal Q_k}) = 0
$$   
and
\begin{equation}
\label{adv3-prime}
  \sum_k \Bigl(\sum_{\mathcal Q_k}\alpha_{\cube}^{\frac{-sc_1}{s-1}}\mu(R)^{\frac{s-p}{(s-1)p}}\Bigr)^{1-s}= \infty.
\end{equation}
\begin{remark}\label{rmk:larger-s-porous}
   Since $1-s<0$, it is easy to see from \eqref{adv3-prime} that every $(s,p)$-porous set $E$, $s>p\ge 1$, is also $(s',p)$-porous for all $s'>s$. This is consistent with Theorem~\ref{jy1} and the general fact that if $E$ is removable for $H^{1,s}(\mathbb R^n;\mu)$, then it is removable for $H^{1,s'}(\mathbb R^n;\mu)$ for all $s'>s$ (Proposition~\ref{prop:q-remov}). 
\end{remark}
\begin{remark}\label{rmk:smaller-p-porous}
    Similarly, every $(s,p)$-porous set $E$, $s>p> 1$, is also $(s,p')$-porous for all $p > p'\ge 1$. Thus, porous sets remain porous if $s$ increases or $p$ decreases (or both).
\end{remark}

In Section~\ref{sec:equiv-por} we give other sufficient conditions for porosity that use (the arguably more natural) $\mu(Q)$ or $\ell(Q)$ rather than $\mu(R)$. But first let us state and prove our main results, which are the re-statements of Theorems A and B from the Introduction.
\begin{theorem}
\label{jy1}
Suppose $p\ge 1$ and $d\mu=wdx$ for a $p$-admissible weight $w$ on $\mathbb R^n, n\ge 2$. If $E\subset \mbbr^n$ is $(s,p)$-porous for some $s>p$, then it is removable for $H^{1,s}(\mathbb R^n;\mu)$.
\end{theorem}

\begin{proof}
Let $\mathcal Q_k,  k=1,\cdots$, denote the collections of cubes given in the definition of porosity for $E$. Throughout, $Q_0$ will be a large cube such that $Q_0$ contains the union of all cubes at all levels. 

Given $ u \in H^{1,s}(Q_0 \setminus E;\mu)$ the goal is to prove that $ u \in H^{1,s}(Q_0;\mu)$. We will first prove that $ u \in H^{1,p}(Q_0;\mu)$, and then improve it to $ u \in H^{1,s}(Q_0;\mu)$.

 By H\"older's inequality $ u \in H^{1,p}(Q_0 \setminus E;\mu) $. Fix $k\in \mathbb N$. For each $Q \in \mathcal Q_k$ its ring $R$ is at a positive distance from $E$, thus $ u \in H^{1,p}(\overline{R};\mu)$. We extend $u|_{R}$ to a function in $H^{1,p}(Q;\mu)$ by Theorem~\ref{extsob}.
 
 Using Proposition~\ref{lem:glue} we glue these extensions to the restriction of $u$ to $Q_0 \setminus \cup \mathcal{Q}_k$ and obtain a function in $H^{1,p}(Q_0;\mu)$. This extension depends on $k$, so, let us denote it by $u_k$. Observe that $u_k$ may differ from $u$ only inside $\cup \mathcal Q_k$.

We will prove that $\{u_k\}_k$ has a Cauchy subsequence in $L^{p}(Q_0;\mu)$, and then we prove that there is a further subsequence such that $\{\nabla u_k\}_k$ is Cauchy in $L^{p}(Q_0;\mu)^n$.

We begin with some bounds on $u_k$. By Theorem~\ref{extsob} and Hölder's inequality, for each $Q \in  \mathcal Q_k$,
\begin{equation}\label{eq:begin}
      \| u_k\|_{L^{p}(Q;\mu)}\lesssim \alpha_{\cube}^{-c_1}\|u\|_{L^{p}(R;\mu)} \lesssim \alpha_{\cube}^{-c_1}(\mu(R))^{\frac{1}{p}-\frac{1}{s}}\|u\|_{L^{s}(R;\mu)}.
\end{equation}
Since the cubes in $\mathcal Q_k$ are pairwise disjoint, applying Hölder's inequality to the discrete sum yields:
\begin{equation}
\label{adv1}
  \|u_k\|_{L^{p}(\cup \mathcal Q_k)}\lesssim \Bigl(\sum_{\mathcal Q_k}\alpha_{\cube}^{\frac{-sc_1}{s-1}}\mu(R)^{\frac{s-p}{(s-1)p}}\Bigr)^{\frac{s-1}{s}}\Bigl(\sum_{\mathcal Q_k}\int_{R}|u|^s d\mu\Bigr)^{\frac{1}{s}}.
\end{equation}
Fix $m>n$. Then,
$$
u_m(x)-u_n(x)=\begin{cases}
            0, \; &\text{if $x \in Q_0 \setminus \cup \mathcal Q_n$}, \\
            u(x)-u_n(x), \; &\text{if $x \in \cup \mathcal Q_n \setminus (\cup \mathcal Q_m)$}, \\
            u_m(x)-u_n(x), \; &\text{if $x \in \cup \mathcal Q_m$}.
        \end{cases}
$$
Therefore, by triangle inequality,
\begin{equation}
\label{adv0}
  \| u_m- u_n\|_{L^p(Q_0;\mu)} \le \| u_m \|_{L^p(\cup \mathcal Q_m;\mu)} + \|u_n\|_{L^p(\cup \mathcal Q_n;\mu)} + \|u\|_{L^p(\cup \mathcal Q_n;\mu)},
\end{equation}
where the correction constant depends only on $p$.

We wish to prove that there is a subsequence such that $\| u_m- u_n\|_{L^p(Q_0;\mu)}$ converges to zero if $m$ and $n$ go to infinity along that subsequence. The last expression on the right-hand side of \eqref{adv0} converges to zero as $n \to \infty$ since $\mu(\cup \mathcal Q_n) \to 0$.

It remains to show that some subsequence of $\|u_n\|_{L^p(\cup \mathcal Q_n;\mu)}$ converges to zero as $n$ goes to infinity; by symmetry the same claim would hold true if $m$ is chosen from the same subsequence. Let us record four observations:
\begin{enumerate}
    \item We can rewrite~\eqref{adv1} as follows (to be justified in a moment)
\begin{equation}
\label{adv1pr}
  \left(\|u_k\|_{L^{p}(\cup \mathcal Q_k;\mu)}\right)^s\lesssim \frac{\sum_{\mathcal Q_k}\int_{R}|u|^s d\mu}{\Bigl(\sum_{\mathcal Q_k}\alpha_{\cube}^{\frac{-sc_1}{s-1}}\mu(R)^{\frac{s-p}{(s-1)p}}\Bigr)^{1-s}}\, \cdot
\end{equation}
\item Because rings are disjoint within and across all levels
\begin{equation}
\label{adv2}
  \sum_k \sum_{\mathcal Q_k}\int_{R}|u|^s d\mu < \infty.
\end{equation}
\item By the $(s,p)$-porosity assumption\footnote{This is how the precise technical porosity requirement was ``discovered''.}
\begin{equation}
\label{adv3}
  \sum_k \left(\sum_{\mathcal Q_k}\alpha_{\cube}^{\frac{-sc_1}{s-1}}\mu(R)^{\frac{s-p}{(s-1)p}}\right)^{1-s}= \infty.
\end{equation}
\item For positively valued sequences,
$$
\text{$\sum_k A_k$ converges and $\sum_k B_k$ diverges $\implies$ $ \liminf_{k\to \infty} \frac{A_k}{B_k} = 0$.}
$$
\end{enumerate}
By applying observations (2)-(4) in (1), we conclude that a subsequence of $\|u_k\|_{L^{p}(\cup \mathcal Q_k)}$ converges to zero. Therefore, if $m$ and $n$ are chosen from this subsequence, we have proven that the first two terms in~\eqref{adv0} converge to zero as $n \to \infty$. The last term also converges to zero, as was discussed above. Therefore, by \eqref{adv0}, $\{u_k\}_k$ has a Cauchy subsequence in $L^{p}(Q_0;\mu)$. Without loss of generality, let us assume $\{u_k\}_k$ itself is Cauchy in $L^p(Q_0;\mu)$.

The fact that a subsequence of $\{\nabla u_k\}_k$ is Cauchy in $L^p(Q_0;\mu)$ can be proved by repeating the above argument almost verbatim, where, beginning with inequality \eqref{eq:begin}, we can replace $u$'s with $\nabla u$'s and $u_k$'s with $\nabla u_k$ and the estimates remain true as they are. We leave out the details.

After passing to a subsequence if needed, $\{u_k\}_k$ is Cauchy in $L^p(Q_0;\mu)$ and  $\{\nabla u_k\}_k$ is Cauchy in $L^p(Q_0;\mu)$,  simultaneously. Therefore,  $\{u_k\}_k$ is Cauchy in $H^{1,p}(Q_0;\mu)$. Denote its limit by $\overline{u} \in H^{1,p}(Q_0;\mu)$. Because for each $k$, we have $u_k=u$ on the complement of $\cup \mathcal{Q}_k$ and $\mu(\cup \mathcal{Q}_k) \to 0$, we see that $\overline{u}=u$ $\mu$-a.e., which proves (Lemma~\ref{cor:a.e.=H1p}) that $u \in H^{1,p}(Q_0;\mu)$. This concludes the first step in the proof.

It remains to prove that $u \in H^{1,s}(Q_0;\mu)$. From the assumptions, $u$ and $|\nabla u|$ are in $L^s(Q_0;\mu)$, and we have just proved that  $u \in H^{1,p}(Q_0;\mu)$. From Proposition~\ref{P1p=H1p} it follows that $u \in H^{1,s}(Q_0;\mu)$.
\end{proof}

\begin{theorem}\label{thm:jyu:-1}
    Suppose $p > 1$ and $d\mu=wdx$ for a $p$-admissible weight $w$ on $\mathbb R^n, n\ge 2$. Then there exists $\eps_0 >0$ such that for every $0< \eps \le \eps_0$, every $(p,p-\eps)$-porous set is removable for $H^{1,p}(\mathbb R^n;\mu)$.
\end{theorem}
\begin{proof}[Proof of Theorem~\ref{thm:jyu:-1}]    
    Let $\eps_0>0$ be such that $\mu$ satisfies a $(p-\eps_0)$-Poincar\'e inequality. Such $\eps_0$ exists due to a deep theorem of Keith-Zhong~\cite{Keith-Zhong:ann}. (Here we may need to appeal to the equality of $H^{1,p}$ with the Newtonian-Sobolev space on the weighted $\mathbb R^n$ \cite[Appendix~A2]{Bj:Bj:11}.)
        
    By remark~\ref{rmk:smaller-p-porous}, every $(p,p-\eps)$-porous set with $0<\eps <\eps_0$ is also a $(p,p-\eps_0)$-porous set. Thus, it will be enough to prove removability of $(p,p-\eps_0)$-porous sets for $H^{1,p}$.
    
    Now if $E$ is a $(p,p-\eps_0)$-porous set then the assumptions of Theorem~\ref{jy1} are satisfied with $s$ taken as $p$ and $p$ taken as $p-\eps_0$. Hence, $E$ is removable for $H^{1,p}(\mathbb R^n;\mu)$.
\end{proof}
As an alternative proof, but still relying on the result of Keith-Zhong~\cite{Keith-Zhong:ann}, we could appeal to the fact (Proposition~\ref{prop:q-admiss}) that every $(p-\eps_0)$-admissible weight is also $(p-\eps)$-admissible for $\eps<\eps_0$, and then apply Theorem~\ref{jy1}.
\section{Annular Decay Property and Porosity}\label{sec:equiv-por}
It is more natural to have a porosity condition in terms of $\mu(Q)$ (better yet $\ell(Q)$) rather than $\mu(R)$. To state such a condition we need to recall some well-known notions regarding doubling measures.

For every doubling measure on $\mathbb R^n$ there exits $0 < \delta \le n$, depending only on $\mu$, such that for every pair of cubes $Q' \subset Q$,
    \begin{equation}
    \label{apeq1-hei}  
    \frac{\mu(Q')}{\mu(Q)} \lesssim \Bigl(\frac{\ell(Q')}{\ell(Q)}\Bigr)^{\delta}.
    \end{equation}
See for example \cite[Chapter~13]{Hei:01}. (To be technically accurate, we can imagine $\mathbb R^n$ is equipped with the infinity norm.)

We decompose the ring into congruent essentially disjoint cubes $Q'$ with $\ell(Q')=\ell(R)/2=\alpha \ell(Q)/2$. The number of such cubes is comparable, with a constant that only depends on $n$, to $(1/\alpha)^{n-1}$. Thus, by applying \eqref{apeq1-hei} we obtain
\begin{equation}
    \label{apeq1}
\mu(R) \lesssim \alpha^{\delta+1-n}\mu(Q).
\end{equation}
The fact that $\delta+1-n \le 1$ is slightly interesting, but \eqref{apeq1} is useless when $\delta+1-n <0$, because it is even weaker than the obvious $\mu(R) \le \mu(Q)$.
\begin{lemma}[\cite{Buckley99}]\label{lem:buck-annular-decay}
    For every doubling measure on $\mathbb R^n, n \ge 2$, there exit constants $0 < \sigma \le 1$ and $C>0$, depending only on $\mu$, such that for any cube $Q$ and its ring $R$ of relative thickness $\alpha$,   we have $    \mu(R) \le C \alpha^{\sigma}\mu(Q).    $
\end{lemma}
To obtain this, one applies \cite[Corollary~2.2]{Buckley99} to $\mathbb R^n$ equipped with the infinity norm. Estimates of this type are very important in analysis and potential theory \cite{Buckley99,bj-bj-leh} and have earned a standard terminology:
\begin{definition}[annular decay property]\label{def:AD}
    We say a doubling measure $\mu$ on $\mathbb R^n$ satisfies the $\sigma$-AD property, where $\sigma > 0$,  if for any cube $Q$ and its ring $R$ of relative thickness $\alpha$, 
    \begin{equation}
    \label{apeq1prime}    
    \mu(R) \lesssim \alpha^{\sigma}\mu(Q).
    \end{equation}
\end{definition}
Clearly, one can always decrease $\sigma$, so we are often interested in the maximal such $\sigma$. Due to \eqref{apeq1} we can, and will, always assume that 
\begin{equation}\label{sigma-big}
    \delta+1-n \le \sigma.
\end{equation}
On the other hand, it is true even in the metric setting \cite{bj-bj-leh} that $\sigma \le 1$. Therefore, $1$-AD property is often called \emph{the strong annular decay property}.

To have a reference point for the parameters, if $\mu$ is the Lebesgue measure, or more generally, if $d\mu=wdx$ where $0<C\le w(x) \le C'<\infty$ for constants $C, C'$, then $\delta = n$ and $\sigma =1$.

We are ready to state our first sufficient condition for porosity.
\begin{proposition}
\label{equivalent-porosity}
 Suppose $d\mu=wdx$ for a $p$-admissible weight $w$, $p\ge 1$, and that $\mu$ satisfies the annular decay property with $0<\sigma\le 1$. Let $E\subset \mathbb R^n$ be a compact $\mu$-null set and $\mathcal Q_k, k=1,2,\cdots$ be a sequence of coverings of $E$ as in the definition of porosity. If for some $s>p$
\begin{equation}
    \label{eq:equiv-poros}
 \sum_k \Bigl(\sum_{\mathcal Q_k}\alpha_{\cube}^{\frac{-sc_1}{s-1}}(\alpha_{\cube}^{\sigma} \mu(Q))^{\frac{s-p}{(s-1)p}}\Bigr)^{1-s}= \infty,
\end{equation}
then $E$ is $(s,p)$-porous.
\end{proposition}
\begin{proof}
Due to $1-s<0$, from \eqref{apeq1prime} we have
$$
\sum_k \Bigl(\sum_{\mathcal Q_k}\alpha_{\cube}^{\frac{-sc_1}{s-1}}\mu(R)^{\frac{s-p}{(s-1)p}}\Bigr)^{1-s} \gtrsim  \sum_k \Bigl(\sum_{\mathcal Q_k}\alpha_{\cube}^{\frac{-sc_1}{s-1}}(\alpha_{\cube}^{\sigma} \mu(Q))^{\frac{s-p}{(s-1)p}}\Bigr)^{1-s}.
$$
Thus, \eqref{eq:equiv-poros} implies that $E$ is $(s,p)$-porous.
\end{proof}
Let $0< \delta \le n$ be the constant that satisfies \eqref{apeq1-hei}, and suppose $\mu$ has the $\sigma$-AD property, i.e.\ \eqref{apeq1prime}. Let $Q_0$ be a fixed cube that contains all of the cubes in the coverings of $E$ for porosity. Then, for each cube in the covering,
\begin{equation}\label{mey4eqfixed}
    \mu(R) \lesssim \alpha_{\cube}^{\sigma}\mu(Q) \lesssim \alpha_{\cube}^{\sigma} (\ell(Q))^\delta, \quad \mu(R) \lesssim \alpha_{\cube}^{\sigma-\delta}(\ell(R))^\delta,
\end{equation}
If we used \eqref{apeq1} instead of $\sigma$-AD we would obtain
\begin{equation}\label{mey4eq-better}
\mu(R) \lesssim \alpha_{\cube}^{\delta+1-n} (\ell(Q))^\delta, \quad \mu(R) \lesssim \alpha^{1-n}(\ell(R))^\delta.
\end{equation}
But recall that we assume, without loss of generality, that $\delta+1-n \le \sigma$, which implies $1-n \le \sigma-\delta$. Thus, \eqref{mey4eqfixed} are always sharper upper bounds than those in \eqref{mey4eq-better}.

Now we can state a sufficient condition for porosity only in terms of lengths:
\begin{proposition}
\label{equivalent-two-porosity}
 Suppose $d\mu=wdx$ for a $p$-admissible weight $w$, $p \ge 1$, and that $\mu$ satisfies \eqref{apeq1-hei} (with $\delta$) and the annular decay property with $0<\sigma \le 1$. Let $E\subset \mathbb R^n$ be a compact $\mu$-null set. Let $\mathcal Q_k, k=1,2,\cdots$ be a sequence of coverings of $E$ as in the definition of porosity. If for some $s>p$
\begin{equation}
    \label{eq:equiv-poros-two}
 \sum_k \Bigl(\sum_{\mathcal Q_k}\alpha_{\cube}^{\frac{-sc_1}{s-1}}(\alpha_{\cube}^{\sigma} (\ell(Q))^\delta)^{\frac{s-p}{(s-1)p}}\Bigr)^{1-s} = \infty,
\end{equation}
then $E$ is $(s,p)$-porous.
\end{proposition}
The proof is similar to that of the previous proposition, so we omit it. Notice that \eqref{eq:equiv-poros-two} also gives an equivalent porosity condition in term of $\ell(R)$ because:
$$
\sum_k \Bigl(\sum_{\mathcal Q_k}\alpha_{\cube}^{\frac{-sc_1}{s-1}}(\alpha_{\cube}^{\sigma} (\ell(Q))^\delta)^{\frac{s-p}{(s-1)p}}\Bigr)^{1-s}=\sum_k \Bigl(\sum_{\mathcal Q_k}\alpha_{\cube}^{\frac{-sc_1}{s-1}}(\alpha^{\sigma-\delta}(\ell(R))^\delta)^{\frac{s-p}{(s-1)p}}\Bigr)^{1-s}.
$$
\section{Examples}\label{sec:examples}
Throughout Section~\ref{sec:examples}, $n=2$ and $d\mu=wdx$ for a $p$-admissible weight $w$, $p\ge 1$, on $\mathbb R^2$.

\subsection{Fat Cantor sets on the \texorpdfstring{$x$}{x}-axis}\label{sec:hyper-cantor-E}
We prove removability for $H^{1,s}(\mathbb R^2;\mu)$ of a class of Cantor sets  lying on the $x$-axis of positive Hausdorff-$1$ measure. We do so by showing  that the Cantor set that we construct is $(s,p)$-porous, hence removable by Theorem~\ref{jy1}.

The construction of $E \subset [0,1]\times \{0\}$ is rather standard~\cite{koskela:hyperplane}. The positive constants $\eta <1$ and $\tau < 1/2$ will be determined below. We begin by deleting an open interval of length $\eta \tau$ from the middle of $[0,1]$. We are left with two \textit{surviving} intervals. This constitutes the stage $k=1$ of the construction. Assuming we have done the construction up to stage $k$, and are left with $2^k$ {surviving} intervals $I_1,\cdots,I_{2^k}$, we proceed to stage $k+1$ by removing an open interval of length \(\eta \tau^{k+1}\) from the middle of each of $I_j, j=1,\cdots,2^k$.

We require that
\begin{equation}\label{eq:eta-tau}
    \eta\tau + 2\eta\tau^2+\cdots = \eta\tau \sum_{i=0}^\infty (2\tau)^i = \eta\tau(1-2\tau)^{-1} < 1.
\end{equation}
Notice that the sum of the lengths of all of the removed intervals from stage $1$ up to stage $k$ is
$\eta\tau + 2\eta\tau^2+\cdots + \eta2^{k-1}\tau^k$, which is strictly less than $1$ by~\ref{eq:eta-tau}. Hence, the construction can be repeated ad infinitum. Define the Cantor set $E$ by
$$
E:=\bigcap_{k=1}^\infty\bigcup_{j=1}^{2^k}I_j.
$$
By condition~\ref{eq:eta-tau}, $E$ will be of positive Lebesgue measure.

We wish to show that there exists a nonempty set of parameters $\eta, \tau, p\ge 1,$ and $s>p$ such that $E$ is $(s,p)$-porous as a subset of $\mathbb R^2$. Denote by $\ell_k$ the common length of the surviving intervals $I_j$ at stage $k$. Using the expression above for the sum of the lengths of all removed intervals up to stage $k$, we have
\[   
{2^k} \ell_k + \eta \tau (1 + 2\tau+\cdots + (2\tau)^{k-1}) =1.
\]
From which, we find
\begin{equation}\label{eq:length-Ij}
    \ell_k = \frac{1-\eta\tau(1-2\tau)^{-1}(1-(2\tau)^k)}{2^k}.
\end{equation}
Fix $k \in \mathbb N$, and let $I_1,\cdots,I_{2^k}$ be the surviving intervals in the construction of $E$. Adjacent to either side of every $I_j$ there are open intervals of length $\eta\tau^{k}/3$ that do not intersect $E$. We consider a cube (=square) $Q_j$ in $\mathbb R^2$ whose horizontal middle segment is the union of $I_j$ and the two intervals to its sides. Clearly, $Q_j$ has a ring $R_j$ of length $\eta\tau^{k}/2$ that is away from $E$. Moreover, $\mathcal Q_k:=\{Q_1,\cdots,Q_{2^k}\}$ are (congruent) pairwise disjoint cubes that cover $E$. (See the $x$-axis in Figure~\ref{fig:product}.)

Let $\alpha_k$ denote the common value of $\alpha_{\cube}$ for all $Q \in \mathcal Q_k$. From \eqref{eq:length-Ij}
\begin{equation}\label{eq:asympt-alpha}
  \lim_{k\to \infty} \frac{\alpha_k}{(2\tau)^k} \quad \text{is a constant.}  
\end{equation}
Now, let $d\mu=wdx$ on $\mathbb R^2$ satisfy \eqref{apeq1-hei} (with $\delta$) and the annular decay property $\sigma$. Recall that, without loss of generality, we assume $\delta-1 \le \sigma \le 1$. From~\eqref{mey4eqfixed} we have
\begin{equation}\label{mey4eq}
    \mu(R_j) \lesssim \alpha_{\cube}^{\sigma-\delta}(\ell(R_j))^\delta \lesssim \alpha_k^{\sigma-\delta}\tau^{k\delta},
\end{equation}
for all $k$ and all $Q_j \in \mathcal Q_k$. The constant $\eta$ is absorbed into the correction factor. It would not play any role in the porosity condition even if we kept track of it.

Let us compute to see if we can find $s>p\ge 1$ for which $E$ will be $(s,p)$-porous, using the cubes described above. It will suffice to achieve 
\begin{equation}\label{divex}
\sum_{k}\Bigl(\sum_{Q\in\mathcal Q_k}
\alpha_{\cube}^{\frac{-sc_1}{s-1}}
\;\mu(R)^{\frac{s-p}{(s-1)p}}\Bigr)^{1-s}
=\infty.
\end{equation}
Recall that $c_1=\log_2 C\ge 1$ is another constant that depends on the weight $w$; it is equal to $\log_2 C_1$, where $C_1$ is identified in Theorem~\ref{extsob}.

By~\eqref{mey4eq}, in order to achieve the divergence \eqref{divex} it is sufficient to have
\begin{equation}\label{divex2}
\sum_{k}\Bigl(\sum_{j=1}^{2^{k}}
\alpha_{k}^{\frac{-sc_1}{s-1}}
\bigl(\alpha_{k}^{\sigma-\delta}\tau^{k\delta}\bigr)^{\frac{s-p}{(s-1)p}}
\Bigr)^{1-s} =\infty.
\end{equation}
In turn, by \eqref{eq:asympt-alpha}, it is enough to have the divergence
\begin{equation}\label{divex3}
\sum_{k}\Bigl(2^k
(2\tau)^{\frac{(-sc_1p+(\sigma-\delta)(s-p))k}{(s-1)p}}\,
\tau^{\frac{k\delta(s-p)}{(s-1)p}}\Bigr)^{1-s}=\infty.
\end{equation}
This is equivalent to
\begin{equation}\label{divex33}
\sum_{k}\Bigl(2^{1-s}(2\tau)^{\frac{sc_1p+(\delta-\sigma)(s-p)}{p}}\,
\tau^{\frac{-\delta(s-p)}{p}}\Bigr)^{k}=\infty.
\end{equation}
This is a geometric sum, thus, to have divergence, we must have
\begin{equation}\label{mey1eq}
    2^{1-s}(2\tau)^{{sc_1}+{(\delta-\sigma)(s-p)}/{p}}\,\tau^{{-\delta(s-p)}/{p}} \ge 1
\end{equation}
For simplicity, write $\tau = 2^{-\upsilon}$. Remember that we must have $\upsilon >1$ so that $\tau <1/2$. Then, condition \eqref{mey1eq} becomes equivalent to
\begin{equation*}\label{mey2eq}
    {(1-s)}+ (1-\upsilon)\Bigl({{sc_1}+\frac{(\delta-\sigma)(s-p)}{p}}\Bigr) + \frac{\upsilon\delta(s-p)}{p} \ge 0.
\end{equation*}

Therefore, we have proved the following:
\begin{proposition}\label{lem:hyper-porous}
Let $E \subset [0,1]\times \{0\}$ be the Cantor set constructed above with the parameters $\tau = 2^{-\upsilon}$, $\upsilon >1$, and $\eta <1$.\footnote{The parameter $\eta<1$ in the construction is irrelevant to the porosity and removability, and the requirement $\eta\tau(1-2\tau)^{-1} < 1$ for the construction can be met for every $\tau$ by choosing $\eta$ smaller.} If for some $s>p$,
    \begin{equation}\label{mey5eq}
    {(1-s)}+ (1-\upsilon)\Bigl({{sc_1}+\frac{(\delta-\sigma)(s-p)}{p}}\Bigr) + \frac{\upsilon\delta(s-p)}{p} \ge 0,
    \end{equation}
then $E$ is $(s,p)$-porous in $\mathbb R^2$, hence, removable for $H^{1,s}(\mathbb R^2;\mu)$. The set of parameters satisfying \eqref{mey5eq} is nonempty whenever $\frac{\delta}{p}>1$.
\end{proposition}
The set of parameters satisfying \eqref{mey5eq} is nonempty when $1\le p < \delta$, because, for example, by choosing $\upsilon$ close to $1$ (note that this means a smaller set $E$) and $s$ large enough, the middle-term is almost zero, and the positive last term over-compensates for the negative first term.
\begin{remark}\label{rem:restrictive}
    The restriction ${\delta}/{p}>1$, although covering many interesting weights, is actually quite restrictive. For example, the rather nice weight $w(x)=|x|^{-1}$ has $\delta = 1$ and so, ${\delta}/{p}>1$ will fail. One possible remedy is to let $\delta=\delta_{\cube}$ depend on $Q$, which is possible because all that one needs is the estimate~\eqref{mey4eq}. Then, if $\delta_{\cube}$ is large for ``most'' cubes in the covering, as would be with $w(x)=|x|^{-1}$, we may still be able to prove applicable sufficient conditions for porosity.
\end{remark}
\subsection{Product sets} We prove that under rather mild conditions on the Cantor set $F \subset \mathbb R^1$, the product set $E \times F$ is porous (hence removable), where $E$ is the Cantor set of the previous subsection constructed with parameters $\tau=2^{-\upsilon}<1/2$ and $\eta <1$, satisfying $\eta\tau(1-2\tau)^{-1} < 1$.

Let us now describe the requirements we will need on $F$. Note that these are sufficient and not necessary conditions for porosity of $E\times F$. Recall that at stage $k=1,2,\cdots$, there are $2^k$ many non-intersecting closed intervals $I_j', j=1,\cdots,2^k$ of length $\ell_k +(2/3)\eta\tau^k$ that cover $E$. The precise value of $\ell_k$ is given in \eqref{eq:length-Ij} but it will not be crucial here.
\begin{figure}
\begin{center}
\begin{tikzpicture}[scale=2]

\def\xshift{0.02}
\def\yshift{0.05}

\draw[-, blue] (\xshift,0) -- (\xshift,2.75) node[above left, blue] {};
\draw[-, blue] (0,\yshift) -- (2.75,\yshift) node[below right, blue] {};

\node[left] at (\xshift,2.57) {1};
\node[below] at (2.55,\yshift-.02) {1};
\node[below] at (-0.02,\yshift-.02 ) {0};

\foreach \x in {0,2,5,7, 15, 17, 20,22} {
    \foreach \y in {0,2,20,22} {
        \draw [pattern=north east lines,
        pattern color=black]
            ({\x/9}, {\y/9}) 
            rectangle ({(\x+1)/9}, {(\y+1)/9});
    }
}
\end{tikzpicture}
\end{center}
\caption{There are $2^k$ intervals covering $E$. But only $\lambda^k$-many intervals, $\lambda<2$,  of the same length are needed to cover $F$. In this picture, $k=2$, $\lambda=\sqrt{2}$.}
    \label{fig:product}
\end{figure}
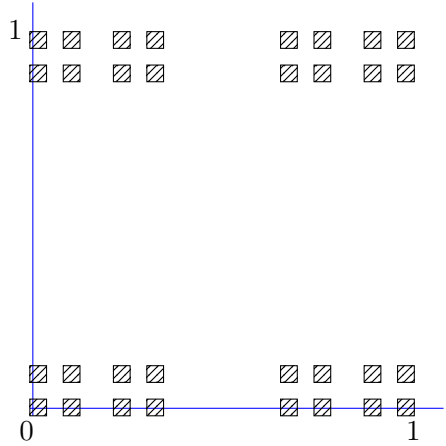

We require that for each $k=1,2,\cdots$ there be a covering of $F$ by (fewer than) $C\lambda^k$-many closed intervals of length $\ell_k$, and that their mutual distance is at least $(2/3)\eta\tau^k$; where $C$ is a constant. Whether $E\times F$ is porous will depend on the constant $\lambda$. (We do not require $F$ to be contained in the unit interval, or any other particular interval.)

Similar to the analysis of the previous section, for each $k=1,2,\cdots$ we cover $E\times F$ by (fewer than) $2^k\lambda^k$-many pairwise disjoint congruent cubes of side length $\ell_k +(2/3)\eta\tau^k$ such that each has a ring of relative thickness asymptotically comparable to $(2\tau)^k$ that are away from $E\times F$.

To test the $(s,p)$-porosity condition for $E\times F$ we run very similar calculations to \eqref{divex}-\eqref{mey1eq}. All that changes is the number of the cubes in \eqref{divex3}. Hence, we end up with the following: The set $E\times F$ is $(s,p)$-porous if
\begin{equation*}\label{eq:ExF-porous1}
    (2\lambda)^{(1-s)}(2\tau)^{{sc_1}+{(\delta-\sigma)(s-p)}/{p}}\,\tau^{{-\delta(s-p)}/{p}} \ge 1.
\end{equation*}
For simplicity, let us write $\lambda = 2^{\omega}$. Therefore, the set $E\times F$ is $(s,p)$-porous if
\begin{equation}\label{eq:ExF-porous2}    
    (1+\omega)(1-s) + (1-\upsilon)\Bigl({{sc_1}+\frac{(\delta-\sigma)(s-p)}{p}}\Bigr) + \frac{\upsilon\delta(s-p)}{p} \ge 0.
\end{equation}
Thus, we have proved the following.
\begin{theorem}\label{thrm:product-porous}
Let $E \subset [0,1]\times \{0\}$ be the Cantor set constructed in Subsection~\ref{sec:hyper-cantor-E} with parameters $\tau = 2^{-\upsilon}$, $\upsilon >1$, and $\eta<1$.  Let $F \subset \mathbb R^1$ be a Cantor set such that for each $k=1,2,\cdots$ there is a covering of $F$ by $C(2^\omega)^k$-many (or fewer) closed intervals of length $\ell_k$, with $\ell_k$ given in \eqref{eq:length-Ij}, and that their mutual distance is at least $(2/3)\eta\tau^k$. If for some $s>p$,
    \begin{equation}\label{mey5eqp}
    (1+\omega)(1-s) + (1-\upsilon)\Bigl({{sc_1}+\frac{(\delta-\sigma)(s-p)}{p}}\Bigr) + \frac{\upsilon\delta(s-p)}{p} \ge 0,
    \end{equation}
then $E\times F$ is $(s,p)$-porous, hence, removable for $H^{1,s}(\mathbb R^2;\mu)$. The set of parameters satisfying \eqref{mey5eqp} is nonempty whenever $\frac{\delta}{p}>1+\omega$.
\end{theorem}
\section{Proofs of the Extension Lemmas}\label{sec:Proof-of-once-reflect}
The proof of removability result Theorem~\ref{jy1} is based on Lemma~\ref{lem:reflect-one-time} and Lemma~\ref{lem:half-Q-ring}. We give the details of a proof of Lemma~\ref{lem:reflect-one-time}. The idea of the proof goes back to~\cite{Jones81}, which was extended to weighted $\mathbb R^n$ in \cite{chua:extend} and \cite{Chua94}. Proof of Lemma~\ref{lem:half-Q-ring} is via slight modifications of the same construction.

Toward the proof of Lemma~\ref{lem:reflect-one-time}, recall that $n\ge 2$, $\alpha <1/2$ and
$$
2R=\{(x_1,\cdots,x_n): (1-2\alpha)\ell/2 < \max |x_i| < \ell/2 \}.
$$
Define
$$
R'=\{(x_1,\cdots,x_n): {(1-2\alpha)}\ell/2 < \max |x_i| < (1-\alpha)\ell/2 \}. 
$$
We begin with a Whitney-type decomposition result whose proof we omit as it can be done quite directly utilizing the particular geometry of the rings.
\begin{lemma}\label{lem:whitn-25}
There exist constants $\kappa >1$, independent of $\alpha$ and $Q$ and a collection $\{D_j\}_{j=1}^\infty$ of closed cubes in $R'$ with the following (with constants independent of $\alpha$, $Q$, $j$, $j'$, $j_0$ and $x$):
    \begin{enumerate}
        \item[(A1)] $ R'\subset \cup_{j}D_j$, and if $j \ne j'$, $\textup{int}(D_j) \cap \textup{int}(D_{j'}) = \emptyset$.
        \item[(A2)] If $\kappa D_j \cap D_{j_0} \ne \emptyset$, then $ \diam D_j \approx \diam D_{j_0}$.
        \item[(A3)] For every $\eps >0$ and every $x_0 \in \overline{R} \cap \overline{R'}$ there exist $r>0$ such that if $x \in R'$, $|x - x_0|<r$, and $D_j \ni x$, then $\diam D_j< \eps$.
        \item[(A4)] For every $x \in R'$, 
        $$
        \sum_j \chi_{\kappa D_j}(x) \lesssim 1.
        $$
    \end{enumerate}
Moreover, for each $D_j$ there corresponds some cube $D_j^*\subset R$ in such a way that:
    \begin{enumerate}
        \item[(B1)] For every $j,$ $\diam (D_j^*) \approx \diam (D_j) \approx \dist(D_j,D_j^*)$. 
        \item[(B2)] If $\kappa D_j \cap D_{j_0} \neq \emptyset$, then $\diam (D_{j}^*) \approx \diam (D_{j_0}^*)$, and there exists a cube $T_{j,j_0} \subset R$ such that $\diam (T_{j,j_0}) \approx \diam (D_{j_0}^*)$ and $ D_j^* \cup D_{j_0}^* \subset T_{j,j_0}$.
        \item[(B3)] Let
        $$
        J_0:=\{j: \kappa D_j \cap D_{j_0} \neq \emptyset\}.
        $$
        For every $x \in R$, 
        $$
        \sum_{j_0}\sum_{j \in J_0}  \chi_{T_{j,j_0}}(x) \lesssim 1.
        $$
    \end{enumerate}
\end{lemma}

We fix $C^1$-functions $0 \le \varphi_j \le 1$ such that 
       $$
        \supp \varphi_j \subset \kappa D_j, \quad |\nabla \varphi_j| \lesssim (\diam D_j)^{-1}, \quad \forall\, x \in R':\, \sum_j \varphi_j(x) = 1.
        $$
From the properties of $D_j$, it follows that every $x \in R'$ has an open neighborhood that intersects the support of only finitely many $\varphi_j$, i.e.\ the sum above is \emph{locally finite}. This allows to differentiate under the summation.

We are ready to use the constructions above to extend functions from $R$ to $2R$. For $u\in \Lip(\overline{R})$, we define
\begin{equation}\label{def:extenstion}
        \overline{u}(x):=\begin{cases}
            u(x), \; &\text{if $x \in \overline{R}$}, \\
            \sum_j u_{\dube_j^*}\varphi_j(x), \; &\text{if $x \in {R'}$}.
        \end{cases}
\end{equation}
\begin{remark}\label{rem:on-Q-indeed}
    Observe that $\overline{u}$ is well defined, by the same formula, on all of $Q$. Indeed, it is Lipschitz on $Q$ as Lemma~\ref{lem:Eu-is-Lip} shows. However, we can prove the desired Sobolev norm bounds only on $R'$; actually on the possibly larger set $\{x: \sum_j \varphi_j(x) = 1\}$ as an inspection of the proof of Lemma~\ref{lem:Eu-bounds:appndx} shows.
\end{remark}
\begin{lemma}\label{lem:Eu-is-Lip}
The function $\overline{u}$ is Lipschitz on $2R=\overline{R} \cup R'$.
\end{lemma}
\begin{proof}
It is clear from (A3) that  $\overline{u}$ is continuous on $\overline{R} \cup R'$. By the local finiteness of the summation, $\overline{u}$ is $C^1$-regular on $R'$. But we now show that $\overline{u}$ is globally Lipschitz on $R'$.

Fix $D_{j_0}$. By the local finiteness of the summation
\begin{equation}\label{eq:march:1}
    \nabla \overline{u}(x)= {\sum_j}u_{\dube_j^*}\nabla \varphi_j(x) = \sum_{j \in J_0}u_{\dube_j^*}\nabla \varphi_j(x), \quad \text{for every $x \in D_{j_0}$}.
\end{equation}
Recall that
$$
J_0= \{j:\kappa D_j \cap \, D_{j_0} \neq \emptyset\}.
$$
Observe from (A4) that $\textup{Card}(J_0) \lesssim 1$.

From the constructions,
\begin{equation}
    \label{eq:night}
    \sum_{j} \varphi_j(x)=\sum_{j \in J_0} \varphi_j(x) = 1, \quad \text{for every $x \in D_{j_0}$}.
\end{equation}
Taking derivative in \eqref{eq:night} yields
$$
\sum_{j \in J_0} u_{\dube_{j_0}^*} \nabla \varphi_j(x)=u_{\dube_{j_0}^*} \nabla \Bigl(\sum_{j \in J_0}  \varphi_j(x)\Bigr) = 0, \quad \text{for every $x \in D_{j_0}$}.
$$
Together with~\eqref{eq:march:1}, this implies that
\begin{equation}
    \label{eq:night2}
    \nabla \overline{u}(x) = {\sum_{j \in J_0}}(u_{\dube_j^*}-u_{\dube_{j_0}^*})\nabla \varphi_j(x), \quad \text{for every $x \in D_{j_0}$}.
\end{equation}
By condition (B2) and Lemma~\ref{prop:avg-difference}
$$
|u_{\dube_j^*}-u_{\dube_{j_0}^*}| \lesssim (\diam T_{j,j_0})\Bigl(\intavg_{T_{j,j_0}} |\nabla u|^p\, d\mu \Bigr)^{1/p} \lesssim (\diam T_{j,j_0}) \|\nabla u\|_{L^\infty(R)}, \; \text{for every $j \in J_0$}.
$$
Therefore, applying $|\nabla \varphi_j(x)| \lesssim (\diam D_j)^{-1}$, (B1), (B2), and $\textup{Card}(J_0) \lesssim 1$ in \eqref{eq:night2} gives
$$
|\nabla \overline{u}(x)| \lesssim \sum_{j \in J_0} \|\nabla {u}\|_{L^\infty(R)} \lesssim \|\nabla {u}\|_{L^\infty(R)} < \infty, \quad \text{for every $x \in D_{j_0}$},
$$
where the last inequality follows from the fact that $u$ is Lipschitz on $R$.

Since, $D_{j_0}$ was arbitrary, we have a global finite bound on $|\nabla \overline{u}|$. Now, it easily follows, for example by integrating $\nabla \overline{u}$ along suitable paths, that $\overline{u}$ is globally Lipschitz on $R'$.

Since $\overline{u}$ is a continuous gluing of two Lipschitz functions on $\overline{R}$ and $R'$ along their common boundary, we have that $\overline{u}$ is Lipschitz on their union.
\end{proof}
\begin{lemma}\label{lem:Eu-bounds:appndx}
    The extension $\overline{u}$ satisfies
    $$
    \|\overline{u}\|_{L^p(R';\mu)} \le C_1\|u\|_{L^p(R;\mu)}, \: \quad \| \nabla \overline{u} \|_{L^p(R';\mu)} \le  C_1 \|\nabla u\|_{L^p(R;\mu)},
    $$
    where $C_1>0$ is a constant that does not depend on $\alpha$ or $Q$ or $u$.
\end{lemma}
\begin{proof}
We begin as in the proof of Lemma~\ref{lem:Eu-is-Lip}, and reach
$$
    \nabla \overline{u}(x) = {\sum_{j \in J_0}}(u_{\dube_j^*}-u_{\dube_{j_0}^*})\nabla \varphi_j(x), \quad \text{for all $x \in D_{j_0}$}.
$$
Recall that $J_0= \{j:\kappa D_j \cap \, D_{j_0} \neq \emptyset\}$ and $\textup{Card}(J_0) \lesssim 1$.

So, by the triangle inequality for the $L^p$-norm
 \begin{align*}
   \|\nabla \overline{u}\|_{L^p(D_{j_0};\mu)} & \le {\sum_{j \in J_0}} |u_{\dube_j^*}-u_{\dube_{j_0}^*}| \|\nabla \varphi_j\|_{L^p(D_{j_0};\mu)} \\
   & \lesssim \sum_{j \in J_0} |u_{\dube_j^*}-u_{\dube_{j_0}^*}|({\diam D_j})^{-1} (\mu(\kappa D_j))^{1/p}.
 \end{align*}
Again, in the proof of Lemma~\ref{lem:Eu-is-Lip} we showed that
$$
|u_{\dube_j^*}-u_{\dube_{j_0}^*}| \lesssim (\diam T_{j,j_0})\Bigl(\intavg_{T_{j,j_0}} |\nabla u|^p\, d\mu \Bigr)^{1/p}.
$$
The doubling of the measure, (B1) and (B2) imply that
$$
\mu(\kappa D_j) \approx \mu(T_{j,j_0}).
$$
Thus,
\begin{align*}
   \|\nabla \overline{u}\|_{L^p(D_{j_0};\mu)} \lesssim \sum_{j \in J_0} \Bigl(\int_{T_{j,j_0}} |\nabla u|^p\, d\mu \Bigr)^{1/p}.
\end{align*}
Since $\textup{Card}(J_0) \lesssim 1$,
$$
\|\nabla \overline{u}\|^p_{L^p(D_{j_0};\mu)} \lesssim \sum_{j \in J_0} \int_{T_{j,j_0}} |\nabla u|^p\, d\mu.
$$
Hence,
 \begin{equation*}\label{eachcube}
     \|\nabla \overline{u}\|^p_{L^p(R';\mu)} \le \sum_{j_0} \|\nabla \overline{u}\|^p_{L^p(D_{j_0};\mu)} \lesssim \sum_{j_0}  \sum_{j \in J_0} \int_{T_{j,j_0}} |\nabla u|^p\, d\mu .
 \end{equation*}
By (B3), this yields
 \begin{equation*}
     \|\nabla \overline{u}\|^p_{L^p(R';\mu)} \lesssim \int_{R} |\nabla u|^p\, d\mu.
 \end{equation*}
We have proved the desired estimate on the norm of gradients, namely,
$$
\|\nabla \overline{u}\|_{L^p(R';\mu)} \lesssim \|\nabla u\|_{L^p(R;\mu)}.
$$
The proof of the remaining estimate, i.e.
 \begin{equation*}
     \|\overline{u}\|_{L^p(R';\mu)} \lesssim \|u\|_{L^p(R;\mu)},
 \end{equation*}
with similar calculations is straightforward and we omit it. Proof of Lemma~\ref{lem:Eu-bounds:appndx} is complete.
\end{proof}
\begin{proof}[Proof of Lemma~\ref{lem:reflect-one-time}]
    For $u \in \Lip(\overline{R})$, let $Eu:=\overline{u}$, with $\overline{u}$ defined in \eqref{def:extenstion}. Clearly, $E$ is linear and in Lemma~\ref{lem:Eu-is-Lip} we proved that $Eu$ is in $\Lip(2R)$. The inequalities in Lemma~\ref{lem:Eu-bounds:appndx} complete the proof of Lemma~\ref{lem:reflect-one-time} with $E$ as the claimed extension operator.
\end{proof}

\begin{proof}[Proof of Lemma~\ref{lem:half-Q-ring}]
As explained in Remark~\ref{rem:on-Q-indeed}, the extension $\overline{u}$ is well-defined on $Q$ and satisfies the inequalities in 
Lemma~\ref{lem:Eu-bounds:appndx} on the set where $\sum_j \varphi_j \equiv 1$. When the ring is large, i.e.\ $\alpha_{\cube}\ge 1/2$, we can modify the cubes of Lemma~\ref{lem:whitn-25} and require further that they cover $Q\setminus \overline{R}$. Then the extension $Eu:=\overline{u}$ satisfies the claims of Lemma~\ref{lem:half-Q-ring}.
\end{proof}
\section{Appendix}\label{sec;appndx}
This section is devoted to the proof of a result that was already used at the end of the proof of Theorem~\ref{jy1}.
\begin{proposition}\label{P1p=H1p}
Fix $p\ge 1$ and $d\mu=wdx$ with a $p$-admissible weight $w$ on $\mathbb R^n$. Suppose $u \in H^{1,p}(Q;\mu)$ for a cube $Q \subset \mathbb R^n$. If $u \in H^{1,p}(Q;\mu)$ and for some $s>p$ we have $u \in L^s(Q;\mu)$ and $|\nabla u| \in L^s(Q;\mu)$, then $u \in H^{1,s}(Q;\mu)$.
\end{proposition}
This answers Question~\ref{areULq} in the affirmative for $p$-admissible weights. While, the analogous claim for $W^{1,p}$ functions would follow immediately from the definitions, for $H^{1,p}$ functions, our proof of Proposition~\ref{P1p=H1p} relies on some techniques from analysis on metric spaces. The proof is inspired by the proofs in \cite{Fr-Haj-Kos} (see also \cite[Theorem~10.3.4]{HKST:15}). In fact, one might be able to deduce Proposition~\ref{P1p=H1p} from \cite[Theorem~10]{Fr-Haj-Kos} which claims that $H^{1,p}$ is equal to the so-called Poincar\'e-Sobolev space $P^{1,p}$ under a suitable Poincar\'e inequality.
\begin{remark}
Proposition~\ref{P1p=H1p} might be true for $p$-admissible \emph{measures} as well, i.e.\ without absolute continuity with respect to the Lebesgue measure, as the proof seems to work in that case as well. Since this would require making sure that that all the lemmata from Section~\ref{sec:prelim} that are appealed to in the proof, explicitly or implicitly, remain valid for $p$-admissible measures, we content ourselves with $p$-admissible weights.
\end{remark}
We will need the fact that a $p$-Poincar\'e inequality implies ``a $(p,p)$-Poincar\'e inequality'' \cite[Theorem~9.1.15, Remark~9.1.19]{HKST:15}.
\begin{remark}\label{reM:N1p=H1p}
    The meaning of a $p$-Poincar\'e inequality in the metric-measure setting of \cite{HKST:15} is very different to the one in this paper, but by the results in \cite[Appendix~A2]{Bj:Bj:11}, our space $H^{1,p}$ coincides with the Newtonian-Sobolev space $N^{1,p}$ on the metric-measure space $(\mathbb R^n,|x-y|,\mu)$. Moreover, the bi-Lipschitz change of the Euclidean norm to the supremum norm, so that the metric balls become cubes, will not nullify the validity of the $(p,p)$-Poincar\'e inequality in Lemma~\ref{pp-Poin}.
\end{remark} 
\begin{lemma}\label{pp-Poin}
If $d\mu=wdx$ for a $p$-admissible weight $w$ on $\mathbb R^n$, then there exists $C>0$ such that    
    $$
    \int_Q |u(x)-u_{\cube}|^p\, d\mu \le C (\diam Q)^p \int_Q |\nabla u|^p\, d\mu,
    $$
for all cubes $Q \subset \mathbb R^n$ and all $u \in H^{1,p}_{loc}(\mathbb R^n)$.
\end{lemma}
\begin{proof}[Proof of Proposition~\ref{P1p=H1p}]
According to Proposition~\ref{prop:H1p-loc+Lp=H1p}, it suffices to prove that $u \in H^{1,s}(\lambda Q;\mu)$ for every $\lambda <1$. Thus, it is sufficient prove the following; let $Q':=(1+\eps)Q$, $\eps >0$.

\textbf{Claim A:} if $u \in H^{1,p}(Q';\mu)$, $u \in L^{s}(Q';\mu)$ and $|\nabla u| \in L^{s}(Q';\mu)$, then $u \in H^{1,s}(Q;\mu)$.

Without loss in generality, assume $\ell(Q)=1$. Fix $r \in \mathbb N$ and divide $Q$ into congruent sub-cubes $Q_k$, each with side-length $r^{-1}$, and with pairwise disjoint interiors. Choose $C^1$-functions $\varphi_k$ such that $\textup{supp}\, \varphi_k \subset (3/2)Q_k $, $|\nabla \varphi_k(x)| \lesssim r$, and $\sum_k \varphi_k(x) \equiv 1$ on $Q$. (Except for finitely many $r$, it will be true that $(3/2)Q_k \in Q'$, so, without loss of generality, we assume this is true for all $r$.)

\textit{1. Construction of an approximating sequence.} We define on $Q$ the function\footnote{Such objects have been called \textit{discrete convolutions} and are standard.}
\begin{equation}
    u_r(x)= \sum_{k} u_{\cube_k}\varphi_{k}(x).
\end{equation}
Clearly, $u_r$ are $C^1$-regular on $Q$, thus, in $H^{1,s}(Q;\mu)$. We will prove that ``basically'' a subsequence of $u_r, r=1,2,\cdots$ converges in $H^{1,s}(Q;\mu)$ to $u$, and this will prove that $u \in H^{1,s}(Q;\mu)$ as desired.

\textit{2. Boundedness of $\|\nabla u_r\|_{L^s(Q;\mu)}$.} We begin by showing that $u_r, r=1,2,\cdots$ is a bounded sequence in  $H^{1,s}(Q;\mu)$. With $r$ still fixed, fix a sub-cube $Q_i$ and let $x$ and $y$ be in $\textup{int}(Q_i)$. Then
\begin{align*}
    |u_r(y)-u_r(x)|&= \Bigl|\sum_{k}u_{\cube_k}(\varphi_k(y)-\varphi_{k}(x))\Bigr|\\
    &= \Bigl|\sum_{k}(u_{\cube_k}-u_{\cube_{i}})(\varphi_k(y)-\varphi_{k}(x))\Bigr|.
\end{align*}
Let $I:=\{k: (3/2)Q_k \cap Q_i \neq \emptyset\}$. So, by properties of $\varphi_k$,
\begin{align}\label{eq:toda1}
    |u_r(y)-u_r(x)| \lesssim  {|x-y|}  \sum_{k \in I}{r}|u_{\cube_k}-u_{\cube_{i}}|.
\end{align}
Choose a cube $\widetilde{Q}_i \subset Q$ such that $Q_i \cup \bigcup_{k \in I} Q_k \subset \widetilde{Q}_i,$ while $\diam \widetilde{Q}_i \approx \diam Q_i$. By Lemma~\ref{prop:avg-difference}
\begin{align}\label{eq:toda2}
|u_{\cube_k}-u_{\cube_{i}}| \lesssim r^{-1}\Bigl(\intavg_{\widetilde{Q}_i}|\nabla u|^p d\mu\Bigr)^{\frac{1}{p}}, \quad \text{for every $k \in I$.}
\end{align}
There is an upper bound, independent of $i$ and $r$, on $\textup{Card}(I)$. Thus, from \eqref{eq:toda1} and \eqref{eq:toda2}
$$
|u_r(y)-u_r(x)| \lesssim  {|x-y|} \Bigl(\intavg_{\widetilde{Q}_i}|\nabla u|^p d\mu\Bigr)^{\frac{1}{p}}.
$$
Remember that $u_r$ are $C^1$-regular, hence differentiable. After dividing by $|x-y|$ in the last inequality and taking limit as $y \to x$, we conclude that 
$$
|\nabla u_r(x)| \lesssim \Bigl(\intavg_{\widetilde{Q}_i}|\nabla u|^p d\mu\Bigr)^{\frac{1}{p}}, \quad \text{for every $x \in \textup{int}(Q_i)$.}
$$
Applying H\"older's inequality to the right-hand side gives
$$
|\nabla u_r(x)| \lesssim \Bigl(\intavg_{\widetilde{Q}_i}|\nabla u|^s d\mu\Bigr)^{\frac{1}{s}}, \quad \text{for every $x \in \textup{int}(Q_i)$.}
$$
Since, $i$ was arbitrary, $\cup \widetilde{Q}_i = Q$, and there is controlled overlap among $\widetilde{Q}_i$, we obtain
\begin{align*}
   \|\nabla u_r\|_{L^s(Q;\mu)} \lesssim \|\nabla u\|_{L^s(Q;\mu)} .
\end{align*}
Since the right-hand side is independent of $r$ and finite, we have proved that $|\nabla u_r|, r=1,2,\cdots$ is bounded in $L^{s}(Q;\mu)$. 

\textit{3. Convergence of $u_r$ in $L^{s}(Q;\mu)$.} 
With notation as above, we have
\begin{align*}
    u_r(x)-u(x) = \sum_{k \in I} (u_{\cube_k}-u(x))\varphi_k(x), \quad \text{for every $i$ and every $x \in \textup{int}(Q_i)$}.
\end{align*}
Hence,
\begin{align}\label{eq:day3}
    \|u_r-u\|_{L^{s}(Q;\mu)} \le \sum_{k \in I} \|u-u_{\cube_k}\|_{L^{s}(Q;\mu)}.
\end{align}
Now, for every $k \in I$ by simple estimates
\begin{align*}
\|u-u_{\cube_k}\|_{L^{s}(Q_i;\mu)}^s &= \int_{Q_i} |u(x)-u_{\cube_k}|^sd\mu(x) \\
                                    &\lesssim \int_{Q_i} \Bigl(|u(x)-u_{\widetilde{\cube}_i}|^s +|u_{\widetilde{\cube}_i}-u_{\cube_k}|^s \Bigr)d\mu(x).
\end{align*}
Thus,
\begin{equation}\label{day:eq4}
  \|u-u_{\cube_k}\|_{L^{s}(Q_i;\mu)}^s \lesssim \int_{\widetilde{Q}_i} |u(x)-u_{\widetilde{\cube}_i}|^s d\mu(x) + \mu(Q_i)|u_{\widetilde{\cube}_i}-u_{\cube_k}|^s.
\end{equation}
Since $w$ is $s$-admissible (Proposition~\ref{prop:q-admiss}), by Lemma~\ref{pp-Poin} 
\begin{equation}\label{day:eq6}
\int_{\widetilde{Q}_i} |u(x)-u_{\widetilde{\cube}_i}|^s d\mu(x) \lesssim r^{-s} \int_{\widetilde{Q}_i} |\nabla u|^{s}\,d\mu, \quad \text{for every $k \in I$}.
\end{equation}
On the other hand, by Lemma~\ref{prop:avg-difference} and H\"older's inequality
$$
|u_{\widetilde{\cube}_i}-u_{\cube_k}| \lesssim r^{-1} \Bigl(\intavg_{\widetilde{Q}_i}|\nabla u|^p d\mu\Bigr)^{\frac{1}{p}} \le r^{-1}\Bigl(\intavg_{\widetilde{Q}_i}|\nabla u|^s d\mu\Bigr)^{\frac{1}{s}}, \quad \text{for every $k \in I$},
$$
hence
\begin{equation}\label{day:eq7}
|u_{\widetilde{\cube}_i}-u_{\cube_k}|^s \lesssim r^{-s}\Bigl(\intavg_{\widetilde{Q}_i}|\nabla u|^s d\mu\Bigr)^{\frac{1}{s}}, \quad \text{for every $k \in I$}.
\end{equation}
Using \eqref{day:eq7} and \eqref{day:eq6} in \eqref{day:eq4} yields
$$
\|u-u_{\cube_k}\|_{L^{s}(Q_i;\mu)}^s \lesssim r^{-s} \int_{\widetilde{Q}_i} |\nabla u|^{s}\,d\mu, \quad \text{for every $k \in I$}.
$$
Since the right-hand side is independent of $k$, applying these inequalities in \eqref{eq:day3} and using $\textup{Card}(I) \lesssim 1$ we get
\begin{align*}
    \|u_r-u\|_{L^{s}(Q;\mu)}^s \lesssim r^{-s}\int_{\widetilde{Q}_i} |\nabla u|^s\,d\mu.
\end{align*}
Summing the last inequality and appealing to the controlled overlap among $\widetilde{Q}_i$ yields
\begin{align*}
    \|u_r-u\|_{L^{s}(Q;\mu)}^s &= \sum_i \|u_r-u\|_{L^{s}(Q_i;\mu)}^s \\
                                &\lesssim r^{-s} \sum_i \int_{\widetilde{Q}_i} |\nabla u|^s\,d\mu\\
                                &\lesssim r^{-s} \int_{Q} |\nabla u|^s\,d\mu.
\end{align*}
Since, the $|\nabla u| \in L^s(Q;\mu)$ by assumptions and the various constants are independent of $r$, this proves that
$$
\lim_{r \to \infty} \|u_r-u\|_{L^{s}(Q;\mu)} = 0.
$$
\textit{4. Convergence in $H^{1,s}(Q;\mu)$.} We have proved that $u_r, r=1,2,\cdots$ is a bounded sequence in $H^{1,s}(Q;\mu)$. By sequential weak compactness, in the sense of Lemma~\ref{lem:weak-conv}, there exists a subsequence that converges weakly in $L^{s}(Q;\mu)$ and whose gradients converge weakly in $L^{s}(Q;\mu)^n$. By Mazur's lemma, a subsequence $v_j, j=1,2,\cdots$, where each $v_j$ is a (finite) convex combination of $u_r$ converges in $H^{1,s}(Q;\mu)$.

In the claim of Mazur's lemma, one can require that the convex combinations use only elements further and further in the tail. Appealing to this, we can further guarantee that $v_j$ converge to $u$ in $L^s(Q;\mu)$, since $u_r$ do. This shows that the limit of $v_j$ agrees $\mu$-a.e.\ with $u$. This proves that $u \in H^{1,s}(Q;\mu)$. Proof of Claim A, hence of Proposition~\ref{P1p=H1p}, is complete.
\end{proof}

\section*{Acknowledgment} We are grateful to Prof.\ Koskela for suggesting the problem and his insightful discussions. We thank Prof.\ Lahti for his remarks on an earlier draft and for pointing out some corrections.

\bibliographystyle{alpha}
\bibliography{Bibliography-Behnam-25}

\end{document}